\documentclass[10pt,twoside,a4paper]{article}
\usepackage{amsmath,amssymb,theorem,euscript,color}

\oddsidemargin

\numberwithin{equation}{section}
\newtheorem{theorem}{Theorem}[section]

\newtheorem{lemma}[theorem]{Lemma}
\newtheorem{proposition}[theorem]{Proposition}

\newenvironment{proof}{{\bf Proof}.\ }{ \hfill $\square$}
\newcommand{\myendproof}{ \hfill $\square$}
\newcommand{\myremark}{ {\it Remark.} \quad}
\newcommand{\bke}[1]{\left ( #1 \right )}
\newcommand{\bkt}[1]{\left [ #1 \right ]}
\newcommand{\bket}[1]{\left \{ #1 \right \}}
\newcommand{\norm}[1]{ \| #1  \|}

\newcommand{\bka}[1]{{\langle #1 \rangle}}
\newcommand{\ip}[2]{{( #1 \, , \, #2 )}}
\newcommand{\abs}[1]{\left | #1 \right |}

\def\al{\alpha}

\def\de{\delta}

\def\ve{\varepsilon}
\def\e {\varepsilon}

\def\th{\theta}
\def\ka{\kappa}
\def\la{\lambda}
\def\si{\sigma}

\def\ph{\varphi} 

\def\om{\omega}
\def\Ga{\Gamma}
\def\De{\Delta}
\def\Th{\Theta}
\def\La{\Lambda}

\def\Om{\Omega}

\newcommand{\R}{\mathbb{R}}

\newcommand{\N}{\mathbb{N}}
\newcommand{\sD}{\EuScript{D}}

\renewcommand{\div}{\mathop{\rm div}}
\newcommand{\curl} {\mathop{\rm curl}}
\newcommand{\supp} {\mathop{\mathrm{supp}}}

\newcommand{\dist} {\mathop{\mathrm{dist}}}

\newcommand{\pd}{\partial}
\newcommand{\nb}{\nabla}
\newcommand{\td}{\tilde}
\newcommand{\wt}[1]{\widetilde {#1}}

\newcommand{\lec}{{\ \lesssim \ }}
\newcommand{\gec}{{\ \gtrsim \ }}

\newcommand{\cD}{\mathcal{D}}
\newcommand{\cE}{\mathcal{E}}
\newcommand{\I}{\infty}

\newcommand{\ot}{\otimes}

\renewcommand{\[}{\begin{equation}}
\renewcommand{\]}{\end{equation}}
\newcommand{\EQ}[1]{\begin{equation}\begin{split} #1 \end{split}\end{equation}}
\newcommand{\Rcut}{\mathcal{R}_{cut}}
\newcommand{\qz}{\frac 3{1+\de}}
\begin{document}

\title{Asymptotics of small exterior Navier-Stokes flows with non-decaying
  boundary data}

\author{Kyungkuen Kang%
\thanks{Department of Mathematics, Yonsei University, Seoul 120-749,
South Korea. Email: kkang@yonsei.ac.kr}, \quad Hideyuki Miura%
\thanks{Department of Mathematics, Osaka University,
Osaka, Japan. Email: miura@math.sci.osaka-u.ac.jp}, \quad Tai-Peng Tsai%
\thanks{Department of Mathematics, University of British Columbia,
Vancouver, BC V6T 1Z2, Canada. Email: ttsai@math.ubc.ca.}}

\date{2011-05-02}

\maketitle

\begin{abstract}
We prove the unique existence of solutions of the 3D incompressible
Navier-Stokes equations in an exterior domain with small non-decaying
boundary data, for $t \in \R$ or $t \in (0,\I)$.  In the latter case
it is coupled with small initial data in weak $L^{3}$. As a corollary,
the unique existence of time-periodic solutions is shown for the small
periodic boundary data.  We next show that the spatial asymptotics of
the periodic solution is given by the same Landau solution at all
times. 
Lastly we show that if the boundary datum is time-periodic and the initial
datum is asymptotically discretely self-similar, then the solution is
asymptotically the sum of a time-periodic vector field and a forward
discretely self-similar vector field as time goes to infinity. It in
particular shows the stability of periodic solutions in a local sense.


\end{abstract}

\section{Introduction}
\label{S1}

Let $\Om\subset \R^3$ be an exterior domain with smooth boundary
$\pd\Om$ and unit outernormal $N$, and $I=\R$ or $I=(0,\I)$ be the
time interval.  In $I\times \Om$ we consider the nonstationary
Navier-Stokes equations
\begin{equation}
\label{NS1} \pd _t u -\De u + (u \cdot \nb) u  + \nb p =f , \quad
\div u = 0,
\end{equation}
\begin{equation}
\label{NS2} u|_{\pd \Om} = u_*, \quad \lim_{|x|\to \infty} u=0,
\end{equation}
where $u(t,x) : \bar I\times \Om\to \R^3$ is the unknown velocity field,
$p(t,x) : \bar I \times \Om \to \R$ the unknown pressure, $u_*$ the given
boundary data and $f=f_0 + \nb \cdot F$ the given force with
$F=(F_{ij})$ being a 2-tensor and $(\nb \cdot F)_j =\pd_i F_{ij}$. We
use summation convention for repeated index and $u \cdot \nb= u_j
\pd_j$.  In the case $I=(0,\I)$, we add the initial condition
\begin{equation}
\label{NS3} u|_{t=0} = u_0,\quad \div u_0=0.
\end{equation}
By an exterior domain we mean a connected open set with bounded
complement.  Without loss of generality, we may assume the complement
of $\bar \Om$ contains the origin and is a subset of $B_{R_1}:=\{x\in
\R^3:|x|<R_1\}$ for some $R_1>0$.

In this paper, we first consider the solvability of the problem
\eqref{NS1}--\eqref{NS3} with nondecaying boundary data in time.
Moreover, we also study the asymptotic properties of solutions with
time-periodic boundary data.

To explain the background, we start with the review of the stationary
problem.  In 1965 Finn \cite{Finn} showed the existence of a small
stationary solution satisfying $|u(x)| \le C|x|^{-1}$ in $\Om$ for
small data.  Nazarov and Pileckas \cite{Naz-Pil} proved that the
solution is asymptotically self-similar at spatial infinity, i.e., the
solution converges to a (-1)-homogeneous vector field faster than
$C|x|^{-1}$.  Recently Korolev and Sverak \cite{Kor-Sve} showed the
asymptotic profile is given by a \textit{Landau solution}.  In
particular the decay rate $|x|^{-1}$ is optimal in general.

Landau solutions are a family of vector fields $U^b$ and functions
$P^b$ in $\R^3$, with parameter $b\in \R^3$, that solve
\begin{equation}
\label{Landau-eq} -\De u + (u \cdot \nb) u + \nb p ={b\de_0}, \quad
\div u = 0,
\end{equation}
where $\de_0$ is the delta function at the origin. They are
axisymmetric: in spherical coordinates $\rho, \th, \varphi$ with
$b$ in the direction of the north pole,
\begin{equation}
{ U^b = \frac 2{\rho}\bke{\frac {A^2-1}{(A-\cos \varphi)^2} -1}
e_\rho + 0 e_\th- \frac {2\sin \varphi} {\rho(A-\cos \varphi)}
e_\varphi,}
\end{equation}
where $A=A(|b|) \in (1,\infty]$ is determined by $ |b| = 16 \pi (A
+ \frac 12 A^2 \log \frac {A-1}{A+1} + \frac {4A}{3(A^2-1)})$ and
strictly decreasing in $|b|\ge 0$.  It is the unique solution of
\eqref{Landau-eq} in the class of ($-1$)-homogeneous vector fields
in $\R^3$. See Landau \cite{Landau,LL}, Tian and Xin
\cite{Tian-Xin}, and Sverak \cite{Sverak}.

On the other hand, there are also a lot of works on the time-periodic
solutions in exterior domains. After the earlier works by Salvi
\cite{Salvi} and Maremonti-Padula \cite{Mar-Pad}, Yamazaki
\cite{Yamazaki} showed the unique existence of time-periodic solutions
in the Lorentz space $L^{3,\I}$ for zero boundary data and small
forces $F \in L^{3/2,\I}$ with $f_0=0$.  Galdi and Sohr
\cite{Galdi-Sohr} further showed the existence of time-periodic
solutions satisfying the pointwise estimate $|u(t,x)| \le C|x|^{-1}$
where $C$ is independent of time.

When we are interested in  time-periodic solutions or,
more generally, the exterior problem for nondecaying
boundary data and forces, the
function spaces should allow nondecaying
functions in time.
On the other hand, in view of the optimal decay rate $|x|^{-1}$
for the stationary solutions, it is natural to choose the spaces
$L^{3,\I}$ (weak-$L^3$ space)
or $X_1$ in spatial variables, where
$X_k$ (for $k>0$) is the space of functions
defined by the norm
\[ 
\label{Xk.def} \norm{u}_{X_k} := \sup_{ x \in \Om}
(1+|x|)^k |u(x)| . 
\]
Let $BC_w(J; X)$ be the class
of bounded and weak-star continuous $X$-valued functions
defined on a time interval $J$. We omit the subscript $_w$ if it
is strongly continuous.

Our first result concerns the unique existence of {\it very weak
  solutions}, to be defined in \S\ref{S2.2}, with small nondecaying
boundary data in the time interval $\R$ or $(0,\I)$. As a consequence,
we obtain an existence theorem of time-periodic solutions.

\begin{theorem}[Existence]
\label{th1} 
Let $\Om$ be a smooth exterior domain in $\R^3$ with $\pd \Om \subset
\{x:|x|<R_1\}$.  Let $I$ be $\R$ or $(0,\I)$. Let $\de$ be a positive
constant.  There are constants $\e_0>\td \e_0>0$ and $C>0$ such that
the following holds: Let $u_0$, $u_*$, $f_0$ and $F$ be given data
with the convention $u_0=0$ if $I=\R$.

(i) Assume that 
\begin{equation} \e:=\norm{u_*}_{W^{1,\I}(I; C^2_x(\pd
\Om))} +\|f_0\|_{L^\infty(I; X_{3+\de})} + \|F\|_{L^{\I}(I;
L^{3/2,\I}(\Om))} +\norm{u_0}_{L^{3,\I} (\Om)} \le \e_0.
\end{equation}
Then there is a
unique very weak solution $u \in BC_w(\bar I; L^{3,\I} (\Om))
$ of \eqref{NS1}--\eqref{NS3} satisfying
\begin{equation}\label{th1-eq1}
\norm{u}_{L^\I(I; L^{3,\I} (\Om))} \le C\e .
\end{equation}

(ii) Let $I=\R$. Assume that
all $u_*$, $f_0$ and $F$ are periodic in time
with  period $T>0$, then the solution in (i) is also periodic in time
with same period.

(iii) Let $I=\R$. Assume that
\begin{equation}\label{iii-est}
 \td \e :=
\norm{u_*}_{W^{2,\I}(\R; C^2_x(\pd \Om))}+\|f_0\|_{W^{1,\I}(\R;
X_{3+\de})} + \|F\|_{W^{1,\I}({\R}; X_{2})}\le \td \e_0,
\end{equation}
then
the solution in (i) satisfies
\begin{equation}\label{th1-eq2}
|u(t,x)| \le C \td \e |x|^{-1}, \quad
(|x|>R_1, \ t\in \R).
\end{equation}
\end{theorem}

{\it Comments for Theorem \ref{th1}:}

\begin{enumerate}

\item
The initial condition \eqref{NS3} is understood in the weak-star sense.
Thus we do not need a compatibility
condition between $u_*$ and $u_0$.
%
When $u_*=0$, we usually also
require $u$ belongs $L^{3,\I}_\si(\Om)$ which is the subspace of
$L^{3,\I}(\Om;\R^3)$ with $\div u=0$ and $u \cdot N|_{\pd \Om}=0$,
$N$ being the outer-normal of $\Om$, see \S\ref{S2.1}.
This is not suitable for the inhomogeneous boundary value problem.


\item In the case $I=\R$ and $u_*$ is time-independent
or time-periodic, we recover the results of Finn
 \cite{Finn}, Galdi and Sohr \cite{Galdi-Sohr} 
on the existence of solutions behaving like
$|x|^{-1}$ as $|x|\to \I$.


\item As noted above, there are a lot of literature on the existence
  of periodic exterior flows for zero boundary data
  \cite{Salvi,Mar-Pad,Yamazaki,Galdi-Sohr}.  Our result allows nonzero
  boundary data.  There are also results on existence of solutions for
  nonzero boundary data e.g.  Amann \cite{Amann}, Farwig, Kozono and
  Sohr \cite{FKS}.  However, most of them require time decay of the
  boundary data, which is not suitable for the periodic solutions.
  Since the {\it maximal regularity estimates} in \cite{Amann, FKS}
  are not available in our solution spaces, we use the duality
  argument by Yamazaki \cite{Yamazaki} to construct the solution.  In
  fact, we first decompose $u=E+v$ where $E$ is an extension of the
  boundary data $u_*$, and then we construct the unique {\it mild
    solution} $v$ of the difference equation \eqref{eq2.72} following
  the method by \cite{Yamazaki}.  Unlike \cite{Yamazaki}, we do not
  require time continuity of the force (in part (i)) and our solution
  is only weak-star continuous in $t$. We then show the equivalence of
  this solution and a very weak solution in \S\ref{sec2.7}.  Note the
  datum $u_*$ is in the $C^2$ class, not usual $C^{2,\al}$, $0<\al<1$.
  The existence of the extension $E$ is shown in Lemma
  \ref{th:extension2}.

\end{enumerate}

Our second result concerns the spatial asymptotics of time-periodic
solutions. To describe it, we recall the {\it momentum flux density
  tensor} for a solution $(u, p)$ of \eqref{NS1},
\begin{equation}
\label{Tij.def} T_{ij}(u,p,F)= p\de_{ij} + u_i u_j - \pd_i u_j -
\pd_j u_i - F_{ij}.
\end{equation}
Equation~\eqref{NS1} can be written as
\begin{equation}
\label{NS1b} \pd_t u _j + \pd_i T_{ij} = f_{0j}, \quad \div u=0.
\end{equation}
We show that the asymptotic profile of a time-periodic solution is
given by a Landau solution determined by the tensor $T_{ij}$ and
chosen independent of time.

%
%

\begin{theorem}[Spatial asymptotics of time-periodic solutions]
\label{th2}
For any $T>0$, $R>0$ and $\al \in (1,2)$, there are
constants $\e_1>0$ and $C>0$
such that the following holds. Suppose $(u,p)$ is a
time-periodic solution  of \eqref{NS1} with period $T$
for $R<|x|<\infty$, and satisfies
\begin{equation}\label{th2-eq1}
\begin{split} \e:= \sup_{t\in \R} \bigg\{&\sup _{|x|>R}
\bket{|x|^{2+\al} |f_0(t,x)| +|x|^{1+\al} |F(t,x)|} +
\norm{u(t,\cdot)}_{L^{3,\I}(|x|>R)} +
\\
&+  \sup _{R<|x|<R+1} \bket{|u(t,x)| +|\nb u(t,x)| +|p(t,x)| }
\bigg \}\quad \le \e_1.\end{split}
\end{equation}
Then $|u(t,x)
| \le C \e |x|^{-1}$ for $|x| >R$ and all $t \in \R$. Moreover,
let $T_{ij}$ be defined by \eqref{Tij.def}, let
\begin{equation}\label{th2-eq2}
b_j = \lim_{\rho \to \I} \frac 1T\int_0^T \int _{|x|=\rho} T_{ij}n_i
dS_x\,dt,\quad (n_i = \frac{x_i}{|x|},\ i=1,2,3),
\end{equation}
and let $U^b$ be the Landau solution corresponding to $b$, whose
choice is independent of $t$. Then
\begin{equation}\label{th2-eq3}
|u(t,x) - U^b(x)| \le C \e |x|^{-\al}, \quad (|x| >R;\ t \in \R).
\end{equation}
\end{theorem}

{\it Comments for Theorem \ref{th2}:}

\begin{enumerate}

\item In Theorem \ref{th2} we do not specify a boundary condition for
$u$. All we need is a solution for $|x|>R$. Thus our result
is applicable to the periodic solutions 
constructed in Theorem \ref{th1} (ii), (iii).


\item If $u$ is independent of time, Theorem \ref{th2} recovers
the result of Korolev-Sverak \cite{Kor-Sve} for small exterior
stationary Navier-Stokes flows.
We observe that the decay rate $|u(t,x)| \le C|x|^{-1}$
for small periodic solutions is optimal as well as
the stationary case.

\item 
In order to see that the limit in \eqref{th2-eq2} exists,
denote the integral as $I_j(\rho)$ and then we have 
from \eqref{NS1b} that 
\begin{equation}
I_j(\rho_2) - I_j(\rho_1) =\frac 1T \int_0^T \int_{\rho_1 < |x| <
\rho_2}  f_{0j} \quad (R<\rho_1 < \rho_2).
\end{equation} 
We would like to emphasize
that the choice of $b$ in \eqref{th2-eq3} is independent of $t$.
There is some sort of cancellation effect behind it.

\end{enumerate}

%
%
%
%

Finally we consider the large time asymptotics of solutions of the
problem \eqref{NS1}--\eqref{NS3} 
when the boundary data and the force are time-periodic.  
Borchers and Miyakawa \cite{BM95} showed the stability of stationary
exterior flows 
under small initial perturbation in $L^{3,\I}$ vanishing at boundary. 
For earlier stability results, see e.g.~\cite{Heywood,KO}.
Our third theorem extends these stability results to 
small time-periodic solutions.

\begin{theorem}[Time asymptotics]\label{th4}
For any $T>0$, $\de> 0$, $\eta>0$ and $3\le q_1 < \frac
3{\de+\eta}$, there is $\e_2>0$ such that
the following holds.
Let $u_*$ and $f$ be time-periodic data satisfying
\eqref{iii-est}.
Assume the initial data $u_0$
is asymptotically self-similar in the sense that
there is a (-1)-homogeneous vector
field $\td u_0 \in X_1$
such that $\div \td u_0=0$ and
\begin{equation}
\label{u_0}
\|\tilde u_0\|_{X_1} \le \ve_2,
\quad \nb \td u_0 \in X_2 \quad \textrm{and} \quad
\td \e :=\norm{u_0-\td u_0}_{L^{3,\I}\cap L^{\qz ,\I}}
 \le \e_2.
\end{equation}
Then the solution $u$ in Theorem \ref{th1} (i) for $t\ge 0$
can be decomposed as
\begin{equation}
\label{dcp0}
u=Q+w+r.
\end{equation}
Here $Q$ is the periodic solution for the data $u_*$ and $f$
in Theorem \ref{th1} (iii). The term
$w$ is the forward self-similar solution of the perturbed Navier-Stokes
system in $(0,\I)  \times \R^3$,
\begin{equation}\label{w-eq0}
\pd_t w -\De w + \nb (w \ot w + U^b \ot w + w \ot U^b) + \nb p_0 =0,
\quad \div w = 0,
\end{equation}
with initial data $w_0=\td u_0-U^b$,
where $U^b(x)$ is the Landau solution corresponding to $Q$
given in Theorem \ref{th2} with $\al=1+\de$, and $w$ satisfies
\begin{equation*}
|w(t,x)| \le C \e_2(|x|+\sqrt t)^{-1+\eta}|x|^{-\eta},
\quad (t \in (0,\I);~ x \in \R^3).
\end{equation*}
The term $r$ satisfies the following decay estimate:
\begin{equation*}
\norm{r(t)}_{L^{q,\I}(\Om)}
\le C \td \ve t^{-\frac32(\frac13 - \frac 1q)-\frac{\de}{2}}
 \qquad  \forall t>0, \quad \forall q \in [3,q_1].
\end{equation*}

\end{theorem}

{\it Comments for Theorem \ref{th4}:}

\begin{enumerate}
\item
The last assumption in \eqref{u_0} means that
the profile of $u_0$ at spatial infinity  is  
given by $\td u_0$.
Since $\norm{Q(t,x)}_{L^{q,\I}}\le C$, 
$\norm{w(t,x)}_{L^{q,\I}}\le C t^{-\frac 32(\frac 13-\frac
1q)}$ for $3\le q \le q_1$ and
 $r$ decays faster than the other terms as $t\to \I$, our theorem
shows that $u$ is asymptotically equal to the sum of $Q$ and $w$.
In particular, this implies the stability of the periodic solution.

\item
In the case $\Om = \R^3$ with 
zero force and no boundary data, Planchon \cite{Pla} showed that 
if the initial data is asymptotically self-similar, 
the solution is also asymptotically self-similar.
Here the asymptotic profile of the solution at large time 
is given by the self-similar solution of 
the non-perturbed equations \eqref{NS1} with zero force.  
See also \cite{CDW}. 
In our case, the asymptotic profile $w$ satisfies an equation  modified by 
the Landau solution $U^b$, but it is still self-similar.
Unique existence of the self-similar solutions for \eqref{w-eq0}
is considered in Proposition \ref{th3}.

\item If $\td u_0$ is not self-similar,
we still have similar decomposition like \eqref{w-eq0}.
In particular, if $u_0$ is asymptotically \textit{discretely
self-similar}, $w$ becomes forward discretely self-similar.
The notion of the discretely self-similar solution
is discussed in the introduction of section \ref{S4}.
In section \ref{S5}, we will show a more general decomposition. 
See Theorem \ref{th5}. 


\end{enumerate}

Our previous three theorems can be similarly posed in the entire
$\R^3$ with a singular force supported at the origin
\begin{equation}
\pd_t u - \De u + (u \cdot \nb) u +\nb p= b(t) \de_0,  \quad \div
u =0, \quad (x \in \R^3),
\end{equation}
which is studied by Cannone and Karch \cite{CK}.
The problem will be pursued elsewhere. 

This paper is organized as follows.  In section \ref{S2}, we recall
some preliminary results related to the exterior problem.  In
section \ref{S2b} we prove Theorem \ref{th1}.  Section \ref{S3} is
devoted to the proof of Theorem \ref{th2}. In section \ref{S4} we
consider the perturbed Navier-Stokes system \eqref{w-eq0} which plays
an important role in the next section.  Finally we consider the time
asymptotics of the exterior flows and prove Theorem \ref{th4} in
section \ref{S5}.

\medskip
{\it Notation}.\quad $a\lec b$ means $a\le Cb$ for some constant
$C$. $a\sim b$ means $a\lec b \lec a$. $\bka{a}=\sqrt{|a|^2+1}$.
For $1\le p\le \I$, its conjugate exponent $p'$ is defined by
$1/p+1/p'=1$. We denote $\R_+ = (0,\I)$.

\section{Preliminaries}
\label{S2}

In this preparation section we recall Helmholtz decomposition and
Stokes semigroup in \S\ref{S2.1}, define very weak solutions in
\S\ref{S2.2}, prove extension lemmas in \S\ref{S2.3}, show that $f_0$
can be absorbed into $\nb F$ in \S\ref{S2.4}, and prove decay estimates
for Stokes system in \S\ref{S2.5}.

Recall $\Om$ denotes an exterior domain in $\R^3$ with smooth boundary
$\pd \Om$ and unit outer-normal $N$.  We also assume $0 \not \in \bar
\Om$ and $\pd \Om \subset B_{R_1}$.

\subsection{Helmholtz decomposition and Stokes operator}
\label{S2.1}

The Helmholtz decomposition
\begin{equation}
L^q(\Om;\R^n) = L^q_\si(\Om) \oplus G^q(\Om), \quad (1<q<\I),
\end{equation}
for a $C^1$-exterior domain $\Om \subset \R^n$, $n \ge 2$, is well
understood. It is first proved
for $n=3$ by Miyakawa \cite{Miyakawa} and
for all $n \ge 2$ by Simader-Sohr \cite{Simader-Sohr}.
Let $P=P_q$ be the associated Helmholtz
projector from $L^q$ onto $L^q_\si$. Then $P$ can be extended by
interpolation to a bounded projector on each Lorentz space
$L^{q,r}(\Om)$, $1<q<\I$, $1\le r \le \I$, with the Helmholtz
decomposition 
\begin{equation}
L^{q,r}(\Om;\R^n) = L^{q,r}_\si(\Om) \oplus G^{q,r}(\Om),
\end{equation}
where \EQ{ L^{q,r}_\si(\Om) &= \{ u \in L^{q,r}(\Om;\R^n) : \div u
=0, u \cdot N|_{\pd \Om}=0\},
\\
G^{q,r}(\Om) &= \{ \nb p \in L^{q,r}(\Om;\R^n) : p \in
L^{q,r}_{loc}(\bar \Om) \}. } Furthermore, if $1\le r<\I$, then
$(L^{q,r}_\si)^* = L^{(q',r')}_\si$ and $( G^{q,r})^* =
G^{(q',r')}$. See \cite[Th.~5.2]{BM95}.

The Stokes operator $A_q$  on $L^q_\si(\Om)$ with the dense domain
\begin{equation}
\sD(A_q) = L^q_\si(\Om) \cap W^{1,q}_0(\Om) \cap W^{2,q}(\Om)
\end{equation}
is defined by $A_q u = - P_q \De u$ for $u \in \sD(A_q)$. It
extends to a closed linear operator on $L^{q,r}_\si(\Om)$ with
domain
\begin{equation}
\sD(A_{q,r}) = \{u \in L^{q,r}_\si(\Om): \nb^j u \in L^{q,r}(\Om),
j=1,2, u|_{\pd \Om}=0\}.
\end{equation}
One also has
\begin{equation}
\sD(A_{q,r}^{1/2}) = L^{q,r}_\si(\Om)\cap W^{1,(q,r)}_0(\Om).
\end{equation}

The semigroup $\{e^{-tA}\}_{t \ge 0}$ also extends to
$L^{q,r}_\si(\Om)$.

We now recall some estimates in Lorentz spaces.

\begin{lemma}
\label{th:2-0A} Let $\Om \subset \R^3$ be an exterior domain with
smooth boundary. One has
\begin{equation} \norm{A^{1/2}u}_{L^{p,r}}\lec \norm{\nb u}_{L^{p,r}},\quad
(1<p<\I, 1\le r \le \I),\end{equation}
\begin{equation}
\norm{\nb u}_{L^{p,\I}}\lec \norm{A^{1/2} u}_{L^{p,\I}}, \quad
(1<p<3),
\end{equation}
\begin{equation}\label{eq2.9}
\norm{\nb u}_{L^{p,1}}\lec \norm{A^{1/2} u}_{L^{p,1}}, \quad
(1<p\le 3).
\end{equation}
\end{lemma}

The above is \cite[Theorem~2.1]{Yamazaki}.

For $1\le p < 3$, we define $p^*$ by $1/p^*=1/p-1/3$.

\begin{lemma}
\label{th:2-1} Let $\Om \subset \R^3$ be an exterior domain with
smooth boundary. 

(i) For $1< p \le q < \I$, $\si = \frac 32(\frac 1p-\frac 1q)$ and $ 1
\le r \le \I$, for any $\phi \in L^{p,r}_\si(\Om)$
\begin{equation}
\label{eq:decay1}
\| e^{-tA} \phi\|_{L^{q,r}} \lec t^{-\si}
\|\phi\|_{L^{p,r}}, \quad t>0,
\end{equation}
\begin{equation}
\label{eq:decay2}
\|A^{1/2} e^{-tA} \phi\|_{L^{q,r}} \lec t^{-\si - \frac 12} \|\phi\|_{L^{p,r}},
\quad t>0.
\end{equation}

(ii) For $1<p< 3$, for some $c=c(\Om,p)>0$, for any $\phi \in
L^{p,1}_\si(\Om)$,
\begin{equation}
\label{eq2.12}
\int_0^\I \norm{ A^{1/2} e^{-tA}  \phi}_{L^{p^*,1}}dt \le c
\norm{\phi}_{L^{p,1}}.
\end{equation}

(iii) For $1<p<3$ and $0<\si<1$, for some
$c=c(\Om,p,\si)>0$, for any $\phi \in L^{p,1}_\si(\Om)$,
\begin{equation}
\label{th:4-1} 
\int^t_0 s^{-\si}
\| A^{1/2} e^{-(t-s)A} \phi\|_{L^{p^*,1}}ds \le c t^{-\si}  \| \phi
\|_{L^{p,1}}, \quad t>0.
\end{equation}
If $1<p\le \frac32$, we may replace $A^{1/2}$ in \eqref{eq2.12}
and \eqref{th:4-1} by $\nb$ due to \eqref{eq2.9}.
\end{lemma}

\begin{proof}
The decay estimates \eqref{eq:decay1} amd \eqref{eq:decay2} are
\cite[(2.4), (2.5)]{Yamazaki}. (Although \cite{Yamazaki} only states
the case $r=1$, it works for other $r$.)  The integral estimate
\eqref{eq2.12} follows from \cite[(2.11)]{Yamazaki}.

To show \eqref{th:4-1}, divide the integral to $s \le t/2$ and $s >
t/2$. Applying \eqref{eq:decay2} for $s \le t/2$ and \eqref{eq2.12}
for $s > t/2$, we get \eqref{th:4-1}.
\end{proof}

\medskip

As a corollary, we have the following lemma. Here $BC_w$ denotes
the class of bounded and weak-star continuous functions.

\begin{lemma}
\label{th:2-1a}
Let $\Om \subset \R^3$ be an exterior domain with
smooth boundary.
Let $I=\R$ or $I=\R_+$. For $f(t) \in L^\I(I; L^{q,
\I}(\Om))$, $3/2 \le q < 3$, define
\begin{equation}
(\mathcal{G}_I f)(t) = \int _{\inf I}^t  e^{-(t-s)A}P\nb f(s) ds, \quad (t
\in I)
\end{equation}
in the sense that, for $p=(q^*)'$, that is, $1/p = 4/3 - 1/q$,
\EQ{\label{GI.wk}((\mathcal{G}_If)(t),\phi)
=
\int_{\inf I}^{t} ( f(s), \nb e^{-(t-s) A}\phi)d s,
 \quad \forall \phi \in L^{p,1}_{\sigma},\ \forall t\in
I.}
Then $\mathcal{G}_I f \in BC_w(\bar I; L^{q^*,\I}_\si(\Om))$, 
and for some
$c=c(\Om,q)>0$
\begin{equation}\label{th:2-1a-eq2}
\norm{\mathcal{G}_I f}_{L^\I(I; L^{q^*, \I}_\si)} \le c \norm{f}_{L^\I(I;
L^{q, \I})}.
\end{equation}
\end{lemma}

Note the integral $\mathcal{G}_I f$ is in weak sense and may not converge
absolutely.

\begin{proof}
For any $\phi \in L^{p,1}_\si$ and $t \in I$,
we have
\EQ{ |(\mathcal{G}_I f(t),\phi)|
& = |\int_{0}^{t-\inf I}
(f(t-\tau),\nb e^{-\tau A}\phi)d \tau|
\\ & \le \int_0^\I \norm{f(t-\tau)}_{L^{q,\I}}
\norm{ \nb e^{-\tau A}\phi}_{L^{q',1}}d\tau . }
Note that $3/2\le q <3$ is equivalent to
$1<p\le 3/2$ and that $q'=p^*$.
Then  Lemma \ref{th:2-1} (ii) implies that
\begin{equation}
|(\mathcal{G}_I f(t),\phi)|
\le c \norm{f}_{L^\I(I;L^{q,\I})} \norm{\phi}_{L^{p,1}},
\end{equation}
which shows $\mathcal{G}_I f \in L^\I(I; L^{q^*, \I}_\si(\Om))$ and
\eqref{th:2-1a-eq2}. To show weak continuity,
a computation similar
to the above shows $\forall t < t' \in I$ and $\e:=t'-t$
\EQ{ & |(\mathcal{G}_I f(t') - \mathcal{G}_I f(t) ,\phi)|
\\
& \le \int_0^\e |(f(t' -\tau), \nb e^{-\tau A}\phi)| d\tau +
\int_0^\I |(f(t-\tau), \nb e^{-\tau A}(e^{-\e A}-1)\phi)| d\tau
\\
& \le \norm{f}_{L^\I L^{q,\I}}
\bke{ \int_0^\e \norm{ \nb
e^{-\tau A} \phi }_{L^{q',1}} d\tau + \int_0^\I  \norm{\nb
e^{-\tau A} (e^{-\e A}-1)\phi}_{L^{q',1}} d\tau}.}
By Lemma \ref{th:2-1} and by the strong continuity at $\ve=0$ of the
Stokes semigroup in $L^{p,1}$, the right hand side converges to 0 as
$\e \to 0_+$, and hence $|(\mathcal{G}_I f(t') - \mathcal{G}_I f(t)
,\phi)| \to 0$ with either $t$ or $t'$ fixed.
\end{proof}

\subsection{Very weak solutions}
\label{S2.2}

In this subsection we define very weak solutions in an exterior domain
$\Om$ with unit outernormal $N$. Our definition is a variant of that
in Amann \cite{Amann} and Farwig, Kozono and Sohr \cite{FKS}. We take
the constant $\ka$ to be $1$ for the Navier-Stokes system, and $\ka=0$
for the Stokes system.

For finite or infinite time interval $J$,
denote the space of test functions
\begin{equation}\label{cDJ.def}
\cD_J := \bket{ w \in C(\Om_J; \R^3):\ 
 \begin{matrix}
\pd_t^a \nb _x ^b w \in
C(\Om_J), (a\le 1; b \le 2); \\
 \div w = 0, \ \supp w \Subset \Om_J,
\ w|_{\pd \Om}=0 \end{matrix} },
\end{equation}
where $\Om_J=\bar J \times \overline{\Om} $ contains the boundary. Thus
$\nb w$ may be nonzero on boundary, and $w(t_0,\cdot)$ may be
nonzero if $t_0=\inf J$ is finite.

 Let $I=\R$ or $I=\R_+$. 
 A vector field $u \in L^2_{loc}(\bar I
\times \bar \Om; \R^3)$ is called a {\it very weak solution} of
the Navier-Stokes system (when $\ka=1$) or of the Stokes system
(when $\ka=0$) with initial datum $u_0$ (with the convention
$u_0=0$ if $I = \R$), boundary datum $u_*$, force $f$ and mass
source $k$ (with sufficient regularity) if
\begin{equation}\label{vws.def}
\begin{split}
&\int_I \bket{- \ip{u}{w_t+\De w}_\Om + \ip{u_*}{N\cdot \nb
w}_{\pd \Om} - \ip{\ka u\ot u}{\nb w}_\Om - \ip{k u}{ w}_\Om} dt \\
&\quad = \ip{u_0}{w(0)}_\Om + \int_I \ip{f}{w}_\Om  dt, \qquad \quad
\forall w \in \cD_{ I}, \end{split}
\end{equation}
and
\begin{equation}\label{vws.def2} \div u(t)=k(t), \quad N \cdot
u(t)|_{\pd \Om} = N \cdot u_*(t) \quad \text{for a.e. } t \in I.
\end{equation}

An elementary calculation shows that for $w\in \cD_J$
\[
N \cdot \nb w = \curl w \times N \quad \text{on }\pd \Om.
\]
Thus \eqref{vws.def} contains a condition only on the tangential
component $N \times u_*$ of $u_*$ on $\pd \Om$, and we have to assume
the additional condition in \eqref{vws.def2} for the normal component
$N \cdot u|_{\pd \Om} = N \cdot u_*$.  Note that, when $\Om$ is a
bounded domain, one needs to assume the compatibility condition
$\int_{\pd \Om} u_*(t)\cdot N dS=\int_\Om k(t)dx$ for a.e.~$t$. When
$\Om$ is an exterior domain, however, this is unnecessary. For the
rest of this paper we take $k=0$.

\subsection{Extension lemmas}
\label{S2.3}

We will use the following extension lemma to extend a given boundary
data $u_*$ with zero flux on every connected component of $\pd \Om$
to a divergence-free vector field of compact support in $\Om$. The
special case we need is $u_*\in C^2(\pd \Om)$ and $\Om$ is a smooth
exterior domain in $\R^3$.

\begin{lemma}
\label{th:extension2} Let $l \in \N$ and $0 \le \al \le 1$. Assume
$\Om$ is  a  domain  in $\R^n$, $n \ge 2$, with compact boundary
$\pd \Om$ of class $C^{l+1,\al}$ and unit outernormal $N$. Suppose $\pd \Om$ has $M$
connected components $\Ga_k$, $k=1,\ldots,M$. For any $\de>0$, there
is a linear map $\cE$ which assigns for each $u_* \in C^{l,\al}(\pd
\Om,\R^n)$ with $\int _{\Gamma_k} u_* \cdot N =0$ for all $k$,  a
vector field $u = \mathcal{E}(u_*)\in C^{l,\al}(\bar \Om)$ so that
\begin{equation}
\label{eq:ext-1} \div u =0, \quad u|_{\pd \Om} = u_*, \quad \supp u
\subset \bar \Om_\de, \quad \norm{u}_{C^{l,\al}(\bar \Om)} \le
C\norm{u_*}_{C^{l,\al}(\pd \Om)}.
\end{equation}
Above $\Om_\de = \{ x \in \Om: \dist(x,\pd \Om)<\de\}$ and
$C=C(\Om,l,\al,\de)$. This linear map restricted to smooth $u_*$ is
the same for all $(l,\al)$ so that $\pd \Om \in C^{l+1,\al}$.
\end{lemma}

The following is a more general result 
which we will not use.
\begin{lemma}
\label{th:extension3} Let $l \in \N$ and $0 \le \al \le 1$.  Assume
$\Om$ is  a  domain  in $\R^n$, $n \ge 2$, with compact boundary
$\pd \Om$ of class $C^{l+1,\al}$ and unit outernormal $N$. For any open convex set $\om
\subset \R^n$ containing $\pd \Om$,  there is a linear map $\cE$ which
assigns for each $u_* \in C^{l,\al}(\pd \Om,\R^n)$ with $\int _{\pd
\Om} u_* \cdot N =0$, a vector field $u = \mathcal{E}(u_*)\in C^{l,\al}(\bar
\Om)$ so that
\begin{equation}
\label{App:eq:ext-1} \div u =0, \quad u|_{\pd \Om} = u_*, \quad
\supp u \subset \bar \Om \cap \om, \quad \norm{u}_{C^{l,\al}(\Om)}
\le C\norm{u_*}_{C^{l,\al}(\pd \Om)}.
\end{equation}
This linear map restricted to smooth $u_*$ is the same for all
$(l,\al)$ so that $\pd \Om \in C^{l+1,\al}$.
\end{lemma}

{\it Remarks on Lemmas \ref{th:extension2} and \ref{th:extension3}:}
\begin{enumerate}
\item If $u_* \in C^{2, \al}(\pd \Om)$, $0< \al <1$,
Lemmas \ref{th:extension2} and \ref{th:extension3} (Lemma
\ref{th:extension3} in the case $\Om$ is bounded) is proved in
Kapitanski\u{i}-Piletskas \cite{Kap-Pil}. 
In addition to induction in dimension, 
they write $u_i = \pd_i (w_{ij} + \de_{ij}\phi)$ with
$w_{ij}$ anti-symmetric and estimate Newtonian potentials in
H\"older spaces. This is why they do not allow $\al=0$ or 
$\al=1$, which are allowed in Lemmas \ref{th:extension2} and
\ref{th:extension3}. 

\item In case $\Om$ is an exterior domain, the extension constructed
in \cite{Kap-Pil} does not have compact support unless one further
assumes $\int _{\Gamma_k} u_* \cdot N =0$ for all $k$.
See also Kozono-Yanagisawa \cite{KY09} for extensions in more general domains.

\item A related problem is the construction of a vector field $v$
for a given $f:\Om \to \R$ satisfying $\int_{\Om}f = 0$ so that
\begin{equation} \label{divvf}
\div v = f,\quad (x\in \Om),\quad v|_{\pd \Om}=0,\quad \norm{\nb v}
\le C\norm{f}
\end{equation}
for suitable norms. The extension problem is reduced to
\eqref{divvf} as follows: first extend $u_*$ to $U \in
C^{k,\al}(\Om)$ which may not be divergence-free, then solve
\eqref{divvf} with $f= \div U$, and finally define $u = U - v$. The
problem \eqref{divvf} is solved first for $f \in C^\I_c(\Om)$ and
then for general $f$ by approximation. For $1<q<\I$, this approach
is good for $u_* \in W^{1-1/q,q}(\pd\Om)$, but not suitable for
$u_*\in W^{k-1/q,q}(\pd\Om)$, $k \ge 2$, since $C^\I_c(\Om)$ is not
dense in $W^{k,q} \cap W^{1,q}_0$. Similarly it is not suitable for
$u_* \in C^2(\pd\Om)$.

\item
Remark VIII.4.1 on \cite[p25]{Galdi2} says that there is no linear
map which assigns a $v$ satisfying \eqref{divvf} for a given $f$ so
that $\norm{v}_{L^q} \le c \norm{f}_{W^{-1,q}}$, which we fully
agree. However, we do not understand why it continues to assert
that, ``by the same token'', there is no linear map which assigns a
$v$ for $f \in C(\bar \Om)$ so that $\norm{v}_{C^1} \le c
\norm{f}_{C}$.

\item Theorem 4 of \cite{Kap-Pil} asserts the existence a solution
$v$ of \eqref{divvf} satisfying (for $0<\al<1$)
\begin{equation}
\norm{v}_{C^{k+1,\al}(\Om)} \le C \norm{f}_{C^{k,\al}(\Om)}.
\end{equation}
However, its proof assumes $\dist(\supp f, \pd \Om)>0$ and the
constant $C$ depends on this distance. See \cite[\S6]{Kap-Pil}.
\end{enumerate}

\medskip

{\it Proof of Lemma \ref{th:extension2}.}\quad We will prove the
case $n=3$ for notational simplicity. The general case ($n \ge 2$)
is proved in the same way. Denote $x'=(x_1,x_2)$ and the open square
$K= \{ x'\in \R^2: -1<x_1,x_2<1\}$.

We first consider the case that the support of $u_*$ is on $K \times
\{0\}$ with $\frac 54 K \times (0,1) \subset \Om$. In this case we
have $\int_K u_*^3 = 0$. Extend $u_*(x')=0$ for $ x'\in \R^2
\backslash K$. Choose $\phi \in C^\I(\R^2)$ so that
$\phi(\xi)=\phi(|\xi|)$,  $\phi(\xi)=0$ for $|\xi|<1/8$ and
$|\xi|>1/4$, and $\int _{\R^{2}}\phi(\xi)d\xi=1$. Also choose $\chi
\in C^\I(\R)$ so that $\chi(s)=1$ for $s<1/8$ and $\chi(s)=0$ for
$s>1/4$. We define the extension by
\EQ{ \label{eq:flat-ext} u(x) = &\Big(  \pd_3 \Phi_1, \quad  \pd_3
\Phi_2, \quad -\pd_1 \Phi_1 - \pd_2 \Phi_2 \Big),}
where\footnote{The formula for $u^3$ is from Lemma 2.5. The formula
for $u^1$ and $u^2$ extends the $n=2$ case of
\cite[(7.3)]{Kap-Pil}.}
\EQ{ \label{eq:flat-ext1} \Phi_j(x',x_3) = \chi(x_3) \bkt{ x_3 \int
_{\R^{2}} u^j(x'+ x_3 \xi) \phi(\xi)d\xi +f^j(x')}, \quad (j=1,2),}
and,  with $g(x_1) = \int_\R u_*^3(x_1,\bar x_2) d\bar x_2$,
supported in $| x_1 |\le 1$ and $\int_{-1}^1 g(x_1)dx_1 =0$,
\begin{equation}
\label{eq:flat-ext2} f^1(x') = -\chi'(x_2)\int_{-\I}^{x_1} g(\bar
x_1) d\bar x_1, \quad f^2(x') = \int_{-\I}^{x_2} u_*^3(x_1,\bar x_2)
d\bar x_2  -  g(x_1) (1-\chi(x_2)).
\end{equation}
(See the proof of Lemma \ref{th:f.dec}.) Note  $f^j$  are supported
in $\bar K$, $\div_{\R^2} f = u_*^3$, and $u$ is supported in $\frac
98 \bar K\times [0,1/4]$. One verifies \eqref{eq:ext-1} directly.

We next consider the case that the support of $u_*$ is on a graph
over $K$: $x_3 = h(x')$ for $x'\in K$ with $|h(x')|<1/4$, and $\Om
\cap (\frac 32K \times [-1,1])$ lies in $x_3>h(x')$. We have
\begin{equation}
N = \frac 1\ell (\pd_1 h, \ \pd_2 h, \ -1), \quad \ell = \sqrt{|\nb
h|^2+1}, \quad dA = \ell dx_1 dx_2,
\end{equation}
and $0=\int_{\pd \Om} u_* \cdot N dA= \int_K u_* \cdot(\pd_1 h, \
\pd_2 h, \ -1) dx_1 dx_2 $. Define new coordinates
\begin{equation}
  y_1 = x_1,\quad  y_2 = x_2,\quad  y_3 = x_3 - h(x'),
\end{equation}
and, for each vector field defined for $x\in \bar \Om \cap (K \times
[0,1])$ define a new vector field for $y\in K \times [-3/4,3/4]$:
\begin{equation}  \label{eq:u-U} U(y) = (u^1(x), \quad u^2(x), \quad u^3(x) -
u^1(x) \pd_1 h(x') - u^2(x) \pd_2 h(x')).
\end{equation}
For a given $u_*$ defined on $\pd \Om$,  we define $U_*$ on $K$ by
the same formula. Then $U_*$ satisfies $\int_K U_*^3 = 0$ and we can
extend $U$ from $U_*$ by the previous case so that $\div U=0$ and
$\supp U \subset K \times [-1/2,1/2]$. We finally define $u$ from
$U$ by \eqref{eq:u-U}. One checks directly that $\div_x u =\div_y U
= 0$.

For the general case, we can find finitely many balls $B_j$,
$j=1,\ldots,J$, with same radius $\rho$ and concentric balls $B_j^*$
with double radius $2\rho$, so that $\pd \Om \subset \cup_{j} B_j$,
$\cup_{j} B_j^* \subset B_{R_1}$, and that $\pd \Om \cap B_j^*$ is a
graph in $B_j^*$ in direction $\mu_j$, and belongs to $|(x-x_j)\cdot
\mu_j| \le \rho/8$ where $x_j$ is the center of $B_j$. Choose a
smooth partition of unity $\{\eta_j\}_j$ on $\pd \Om$ so that
$\sum_j \eta_j = 1$ on $\pd \Om$ and $\supp \eta_j \subset B_j \cap
\pd \Om$.

For any given $u_*$ with $\int_{\Ga_k} u_*\cdot N=0$ for all $k$, we
claim we can decompose
\begin{equation}
u_* = \sum_j u_{*,j}, \quad \supp u_{*,j} \subset \pd \Om \cap B_j,
\quad \int _{\pd \Om} u_{*,j} \cdot N=0,
\end{equation}
with suitable estimates. When $\pd \Om$ has only one component
($M=1$), choose $\chi_k$, $k=1, \ldots, J-1$, so that $\chi_k \in
C^\I_c( \pd\Om \cap B_k \cap \cup_{j>k} B_j)$ with $\int_{\pd \Om}
\chi_k = 1$. We define $u_{*,j}$ by induction: Let $U_1 = u_*$, and
for $j=1,\ldots,J-1$,
\begin{equation}
u_{*,j} = U_{j}\eta_j - \chi_j \int_{\pd \Om} U_{j}\eta_j, \quad
U_{j+1} = U_{j} -u_{*,j}
\end{equation}
and $u_{*,J} = U_J$. When $\pd \Om$ has more than one component, we
can perform the above decomposition for each component. (The above
decomposition follows the proof of \cite[Lemma III.3.2]{Galdi}.)

Since $\int _{\pd \Om} u_{*,j} \cdot N=0$, by the second case above
with a suitable rescaling we can extend $u_{*,j}$ to divergence-free
$u_j$ supported in $B_j^*\cap \bar \Om_\de$. We now define $\mathcal{E}(u_*) =
\sum_j u_j$. \myendproof

\medskip

{\it Remark}. \quad If $u_*$ is prescribed in $K\times \{0,1\}$ with
$u_*(x',0)= (0,0,v_0(x'))$,  $u_*(x',1)= (0,0,v_1(x'))$ and
$\int_{K} v_0(x')dx' =\int_{K} v_1(x')dx'$, we can extend $u_*$ to
$K\times [0,1]$ by
\EQ{ u(x) &= (\psi'(x_3) f^1(x'),\ \psi'(x_3) f^2(x'), \
(1-\psi(x_3)) v_0(x') + \psi(x_3)v_1(x')),}
where $\psi = 1 -\chi$,
\begin{equation}
f^1(x') = \psi'(x_2)\int_{-\I}^{x_1} g(\bar x_1) d\bar x_1, \quad
f^2(x') = \int_{-\I}^{x_2} (v_0-v_1)(x_1,\bar x_2) d\bar x_2  -
g(x_1) \psi(x_2),
\end{equation}
and $g(x_1) = \int_\R (v_0-v_1)(x_1,s) ds$.

\medskip

{\it Proof of Lemma \ref{th:extension3}}. \quad  The case $M =1$
follows from Lemma \ref{th:extension2}, with $\de>0$ chosen so small
that $\Om_\de \subset \om$. Suppose now $M \ge 2$.

Claim: We can choose $M-1$ line segments $L_k$ connecting $\Ga_j$,
so that $L_k$ intersect $\pd \Om$ only at end points and at right
angles, and each $\Ga_j$ intersects at least one $L_k$. This is
chosen by induction. Let $A_1=\{ 1\}$ and $A_{M} = \{ 1, \ldots,
M\}$. Suppose $A_k$, a subset of $A_M$ has been chosen for $k=1,
\ldots, M-1$. Since $\pd \Om$ is compact, one can find a line
segment $L_{k}$ minimizing the distance
\begin{equation}
\dist( \cup_{j \in A_k} \Ga_j, \, \cup_{j
\not \in A_k}\Ga_j ).
\end{equation}
Clearly $L_k$ intersects $\cup_{j \in A_k} \Ga_j$ only at one
endpoint and $L_k$ intersects $\cup_{j \not \in A_k} \Ga_j$, say
$\Ga_{j_{k+1}}$, at the other endpoint. Moreover, $L_k$ without
endpoints is inside $\Om$. Now let $A_{k+1} = A_k \cup \{ j_{k+1}\}$
and continue.

Once all $L_k$ have been chosen, we give a rank of $\Ga_j$ and $L_k$
as follows: $\Ga_j$ is assigned rank one if it intersects only one
$L_k$. This line segment is renumbered as $L_j$ and also assigned
rank 1. Let $A$ be the set of $\Ga_j$ and $L_k$ without those of
rank one. Rank one boundaries and line segments in this reduced set
are assigned rank 2. We continue this exercise until we exhaust all
$\Ga_j$ and $L_k$.

Fix $0<\e\ll 1$. For each $k$, let $e_k$ be the unit direction
vector of $L_k$ (unique up to a negative sign), let $L_k^\e$ be
those points on the line extending $L_k$ with distance to $L_k$ less
than $\e$, let $r_k(x)=\dist(x, L_k^\e)$, and let $ U^k(x) =\e^{1-n}
\chi(\e^{-1} r_k(x)) e_k$ where $\chi: \R \to [0,1]$ is a fixed
smooth cut-off function with $\chi(r)=1$ for $r<1$ and $\chi(r)=0$
for $r>2$. For $\e>0$ sufficiently small (depending on $\pd \Om$),
we have $\int _{\Ga_j} U^k \cdot N = \pm (\int_{\R^{n-1}}
\chi(|x'|)dx' + o(1)) \not =0$ if $\Ga_j$ is one of the two
boundaries intersecting $L_k$. We also have $\supp U^k \subset \om$.

Now, for a given $u_*$, let $u_*^0 = u_*$ and define recursively for
$k \ge 1$:
\begin{equation} u_*^k(x) = u_*^{k-1}(x) - \sum_{\Ga_j\text{ of rank }k} c_jU^j(x), \quad
c_j =  \frac {\int_{\Ga_j} u_*^{k-1} \cdot N}{\int_{\Ga_j} U^j \cdot
N}.
\end{equation}
Denote the final one as $u_*^M$. One verifies that $\int
_{\Ga_k}u_*^M \cdot N=0$ for all $k$ and
$\norm{u_*^M}_{C^{l,\al}}+\sum_j |c_j| \le
C\norm{u_*^M}_{C^{l,\al}}$. We now define
\begin{equation}
\mathcal{E}(u_*) = \mathcal{E}(u_*^M) + \sum_{j=1}^M c_jU^j(x),
\end{equation}
where $\mathcal{E}(u_*^M)$ is defined by Lemma \ref{th:extension2}. One
verifies that $\norm{\mathcal{E}(u_*)}_{C^{l,\al}(\Om)}\le C
\norm{u_*^M}_{C^{l,\al}}+C\sum_j |c_j| \le
C\norm{u_*^M}_{C^{l,\al}}$, and that the support of 
$\mathcal{E}(u_*)$ is
inside $\om$. \myendproof

\subsection{Source terms}
\label{S2.4}

In this subsection we show that any force $f(x)$ in $\R^n$ decaying
like $|x|^{-n-\e}$ as infinity can be decomposed in the form $f=f_0+
\nb \cdot F$ with $\supp f_0$ being compact.

\begin{lemma}\label{th:f.dec}
   If $f(x)$ is defined in $\R^n$ with $|f(x)| \lec
\bka{x}^{-a}$, $a>n\ge 1$, then for any $R>0$ we can rewrite
\begin{equation}
\label{f.dec} f(x) = f_0(x) + \sum_{j=1}^n \pd_j F_j(x)
\end{equation}
where $\supp f_0\in B_R(0)$ and $|F_j(x)| \lec
\bka{x}^{-a+1}\norm{\bka{x}^a f(x)}_{L^\I}$.
\end{lemma}

\myremark (i) If  $f(x)$ is defined in an exterior domain $\Om
\subset \R^n$ with $0 \not \in \bar \Om$, we may extend $f$ by zero
to entire $\R^n$, and choose $R>0$ so small that $\Om \cap B_R =
\emptyset$. Then the first term $f_0(x)$ in the decomposition
\eqref{f.dec} can be ignored for $x \in \Om$.

(ii) If $f = g_0 + \nb \cdot G$ with $|g_0(x)| \lec \bka{x}^{-a}$,
$a>n\ge 1$, and $|G(x)| \lec \bka{x}^{-a+1}$, we may assume $g_0 $
has compact support by decomposing $g_{0i} =f_{0i}(x) + \sum_{j=1}^n
\pd_j F_{ji}(x)$ as in the lemma and absorbing $F$ to $G$. If $f$ is
defined in an exterior domain we may assume $g_0=0$ by (i) above.

\begin{proof}
We will prove the case $n=3$. The proof for the general case is
similar. By rescaling we may assume $R=1$ and $\norm{\bka{x}^a
f(x)}_{L^\I}=1$. We first consider the case $\supp f \subset B_2$.
Choose a smooth $\psi(t):\R \to \R$, $\psi(t)=1$ for
$t>R_1=3^{-1/2}$ and $\psi(t)=0$ for $t<-R_1$. The region $\{x:
|x_j|<R_1\} \subset B_1$. Define
\begin{equation}
G_3(x) = \int_{-\I}^{x_3} f(x_1,x_2,\bar x_3)\,d\bar x_3, \quad
I_3(x_1,x_2) = \lim_{x_3\to +\I} G_3(x),
\end{equation}
\begin{equation}
F_3(x) = G_3(x) - \psi(x_3) I_3(x_1,x_2).
\end{equation}
Then  $f = \pd_3 F_3 + \psi'(x_3)I_3(x_1,x_2)$. Define
\begin{equation}
G_2(x_1,x_2) = \int_{-\I}^{x_2} I_3(x_1,\bar x_2)\,d\bar x_2, \quad
I_2(x_1) = \lim_{x_2\to +\I} G_2(x_1,x_2),
\end{equation}
\begin{equation}
F_2(x) = \psi'(x_3)[G_2(x_1,x_2) - \psi(x_2) I_2(x_1)].
\end{equation}
Then $f = \pd_3 F_3 + \pd_2 F_2+ \psi'(x_2)\psi'(x_3)I_2(x_1)$.
Define
\begin{equation}
G_1(x_1) = \int_{-\I}^{x_1} I_2(\bar x_1)\,d\bar x_1, \quad I =
\lim_{x_1\to +\I} G_1(x_1) = \int_{\R^3} f,
\end{equation}
\begin{equation}
F_1(x) = \psi'(x_2)\psi'(x_3)[G_1(x_1) - \psi(x_1) I].
\end{equation}
Then
\begin{equation}
f = \pd_3 F_3 + \pd_2 F_2+\pd_1 F_1+ \psi'(x_1)
\psi'(x_2)\psi'(x_3)I.
\end{equation}
The last term vanishes if $\int f=0$. Also note $F_j$ are at least
as regular as $f$,
\begin{equation}
|F_j(x)|\lec \|f\|_{L^\I}, \quad \text{and }\supp F_j\subset \{x:
|x_j|<2\}.
\end{equation}

For the general case, choose  a smooth $\ph(x)$ supported in $\{x\in
\R^3:2^{-1} \le |x| \le 2\}$ satisfying $\ph(x)>0$ for $2^{-1} < |x|
< 2$ and $\sum_{k=-\I}^\I \ph(2^{-k}x) = 1$ for $x\not =0$. Let
$\ph_k(x) =\ph(2^{-k}x)$ for $k\ge 0$ and $\psi(x) = 1-
\sum_{k=1}^\I \ph(2^{-k}x) $. Define
\begin{equation}
a_k = \frac 1{2^k \int \ph}\int f \sum_{j>k}\ph_j, \quad (k \ge 0).
\end{equation}
Then $|a_k| \lec 2^{-ka}$. Decompose
\begin{equation}
f=\sum_{k=0}^\infty f_k, \quad f_0= f\psi + a_{0}\ph_0, \quad f_k=
(f + a_{k})\ph_k -a_{k-1}\ph_{k-1}, \quad (k \ge 1) .
\end{equation}
One verifies that
\begin{equation}
\int f_0 =\int f,   \quad \int f_k = 0, \quad (k \ge 1) ,
\end{equation}
$\supp f_k \subset \{x: 2^{k-2} \le |x|\le 2^{k+1}\}$, and
$|f_k(x)|\le C \bka{x}^{-a}$. The rescaled function $\td f_k(x) =
f_k(2^k x)$ is supported in $B_2$. The previous case gives the
existence of $F_{kj}$, $j=1,2,3$, so that
\begin{equation}
f_k = \sum_{j=1}^3 \pd_j F_{kj}, \quad |F_{kj}(x)|\le C
\bka{x}^{-a+1}.
\end{equation}
Thus \begin{equation}f=f_0 + \sum_{j=1}^3 \pd_j F_j, \quad F_j =
\sum_{k=1}^\I F_{kj}, \quad |F_j(x)| \le \sum_{k \gec \ln |x|}
2^{-(a-1)k} \lec \bka{x}^{-a+1}.
\end{equation}
This shows the lemma.
\end{proof}

\subsection{Stokes system}
\label{S2.5}

In this subsection we proved decay estimates for the Stokes system
in $\R \times \R^3$. Recall the fundamental solutions of the Stokes
system (see \cite{Oseen} and \cite[page 27]{Solonnikov})
\begin{equation}
\label{Stokes-tensor}
S_{ij}(t,x)=\Gamma(t,x)\delta_{ij}+\frac{1}{4\pi}\frac{\partial^2}{\partial
x_i\partial x_j}\int_{\R^3}\frac{\Gamma(t,y)}{\abs{x-y}}dy,\qquad
Q_j(t,x)=\frac{\delta(t)}{4\pi}\frac{x_j}{\abs{x}^3},
\end{equation}
where $\Gamma$ is the fundamental solution of the heat equation. It
is known in \cite[Theorem 1]{Solonnikov} that the tensor
$S=(S_{ij})$ satisfies the following estimates:
\begin{equation}\label{estimate-T}
\abs{D^\ell_x\pd^k_t S(t,x)}\leq C_{k,l}
(|x|+\sqrt{t})^{-3-\ell-2k}, \quad (\ell,k \ge 0),
\end{equation}
where $D^\ell_x$ indicates $\ell$-th order derivatives with respect
to the variable $x$.

A solution of the non-stationary Stokes system in $\R^3$,
\begin{equation}\label{Stokes}
\pd_t w - \De w+ \nb p =g + \nb G, \quad \div w =0, \quad ((t,x) \in
\R^{1+3}),
\end{equation}
if $g$ and $G$ have sufficient decay, is given by
\begin{equation}
w = \La(g) + \Th(G),
\end{equation}
where
\begin{align}
\label{Phi-def} (\La
g)_i(t,x)&=\int_0^{\infty}\int_{\R^3}S_{ij}(s,x-y)
g_{j}(y,t-s)dyds. \\
\label{Psi-def} (\Th G)_i(t,x) &
=-\int_0^{\infty}\int_{\R^3}\partial_{k}S_{ij}(s,x-y)
G_{jk}(y,t-s)dyds. \end{align}

\begin{lemma} \label{th:Stokes}

\begin{itemize}
\item[(i)]\,\, Suppose $G \in L^\I X_{\al+1}$, $0<\al<2$.
Then  $\Th G \in L^\I X_\al$ with $\norm{\Th G}_{L^\I X_\al} \le
C_\al \norm{G}_{L^\I X_{\al+1}}$.
\item[(ii)]\,\, Suppose $g \in L^\I X_{\al+2}$, $1<\al$. Then  $\La g \in L^\I X_{1}$ with $\norm{\La g}_{L^\I
X_{1}} \le C_\al \norm{g}_{L^\I X_{\al+2}}$.

\item[(iii)]\,\, Suppose  $g \in L^\I X_{\al+2}$, $1<\al$, $g$ is time periodic of
period $T>0$, and $\int_0^T \int g dx \,dt=0$. Then  $\La g \in L^\I
X_{2}$ with $\norm{\La g}_{L^\I X_{2}} \le C_\al \norm{g}_{L^\I
X_{\al+2}}$.

\end{itemize}
\end{lemma}

Note that time periodicity is only assumed in (iii).

\begin{proof}
(i) For $\Th G$ defined by \eqref{Psi-def}, using the estimate
\eqref{estimate-T} and integrating in time, we obtain
\begin{equation}
\abs{\Theta G(t,x)}\leq \int_0^{\infty}\int_{\R^3}
\frac{C}{(\abs{x-y}+\sqrt{s})^{4}} \bka{y}^{-\alpha-1} dyds \lec
I_\al(x),
\end{equation}
where
\begin{equation}
I_\al(x)= \int_{\R^3}\frac{1}{\abs{x-y}^{2}}\bka{y}^{-\alpha-1}dy,
\quad (\al>0).
\end{equation}
If $|x|<10$, then $I_\al(x)\lec 1$. If $|x|\ge 10$, then
\begin{equation} I_\al(x) \le \int_{|y|<1}\frac {dy}{|x|^2} + \int
\frac {dy}{|x-y|^2 |y|^{\al+1}} = C|x|^{-2} + C |x|^{-\al},
\end{equation}
if $0<\alpha<2$. We conclude $I_\al(x) \le C_\al \bka{x}^{-\alpha}$
and the estimate for $\Th G$.

\medskip

(ii)  Due to Lemma \ref{th:f.dec}, $g$ can be decomposed as
$g=g_0+\nabla \tilde{G}$, where ${\rm{supp}}\,g_0\in B_1(0)$ and
$\norm{\tilde{G}}_{L^{\infty} X_{\al+1}}\lec \norm{g}_{L^{\infty}
X_{\al+2}}$. Since $\nabla \tilde{G}$ can be treated as in the case
(i), we consider the only the case that $g_0\neq 0$ and
$\tilde{G}=0$. We may assume $\norm{g_0}_{L^\I X_{\al+2}}\le 1$. By
Young's convolution inequality
\begin{equation}\label{w.Linfty}
\begin{split}
   |\La g(t,x)| &\lec \int_0^{\infty} \norm{S_{ij}(s,\cdot)}_{L^2 +
L^\infty} \norm{ g_{j}(\cdot,t-s)}_{L^2 \cap L^1}ds
\\
&\lec \int_0^\infty \min (s^{-3/4}, s^{-3/2}) ds \lec 1.
\end{split}
\end{equation}
For $|x|>2$ we have
\begin{equation}
\begin{split}
\abs{\La g(t,x)} &\leq \int_0^{\infty}\int_{\abs{y}\leq 1}\abs{
S_{ij} (x- y, s)}\abs{g_j(y,t-s)}dyds
\\
&\leq C\int_0^{\infty}\int_{\abs{y}\leq 1}\frac{1}{(\abs{x
}+\sqrt{s})^3}dyds\leq
C\int_0^{\infty}\frac{1}{(\abs{x}+\sqrt{s})^3}ds\leq
\frac{C}{\abs{x}}.
\end{split}
\end{equation}

(iii) Continue part (ii) and assume $|x|>2$. Using $\int_0^T \int
g(t,y)dydt=0$, we have
\begin{equation}
(\La g)_i(t,x) = \int_0^{\infty}\int_{\R^3}
\bke{S_{ij}(s,x-y)-S_{ij}(s,x)} g_{j}(y,t-s)dyds.
\end{equation}
Using the mean-value formula and \eqref{estimate-T},
\begin{equation}
\begin{split}
\abs{(\La g)(t,x)} &\leq \int_0^{\infty}\int_{\abs{y}\leq
1}\abs{\nabla_y S_{ij} (s,x-\theta y)\cdot y}\abs{g_j(t-s,y)}dyds
\\
&\leq C\int_0^{\infty}\int_{\abs{y}\leq 1}\frac{1}{(\abs{x-\theta
y}+\sqrt{s})^4}dyds\leq
C\int_0^{\infty}\frac{1}{(\abs{x}+\sqrt{s})^4}ds\leq
\frac{C}{\abs{x}^2},
\end{split}
\end{equation}
where $\theta=\th(x,y,s)\in [0,1]$, and we have used that $g_j$ is
bounded and supported in $\abs{y}\leq 1$ and $\abs{x}>2 $. This
shows $\norm{\Lambda g}_{L^\I X_2}\le C_\al \norm{g}_{L^\I X_{\al+2}}$ and
completes the proof.
\end{proof}

\section{Existence of flows with non-decaying boundary data}
\label{S2b}

In this section we prove Theorem \ref{th1}. By Lemma \ref{th:f.dec},
we may assume $f_0=0$ by absorbing $f_0$ into $F$.

\subsection{Construction of a mild solution}
\label{sec2.6}

{\it Proof of part (i)}.\quad We first consider the case $I = \R$
and denote
\begin{equation}
\mathcal{X}^q =L^\I(\R, L^{q,\I}).
\end{equation}

Let $\Ga_k$, $k=1,\ldots, M$, be the connected components of $\pd
\Om$, and  choose $x_k$ in the bounded open region enclosed by
$\Ga_k$. Let $H_0 (x) =\frac 1{|x|} $.
Note
\begin{equation}
\label{V0-bdint} \int_{\Ga_k} \nabla H_0(x-x_l) \cdot N = 4\pi\de_{kl},
\quad \forall k,l.
\end{equation}
For each $t\in I$, let
\begin{equation}
\label{H-def}
H(t,x) = \sum g_k(t) H_0(x-x_k), \quad g_k(t) =
\frac 1{4\pi}\int_{\Ga_k} u_*(t) \cdot N.
\end{equation}
Note $H$ is harmonic in $x\in \Om$.  Denote $u_*^\shortparallel(t)
=u_*(t) - \nb H(t)|_{\pd \Om}$.  By \eqref{V0-bdint} and
\eqref{H-def},
\begin{equation}
\int_{\Ga_k} u_*^\shortparallel(t) \cdot N = 0, \quad \forall k,\
\forall t.
\end{equation}
By Lemma \ref{th:extension2}, we can define $E_1=
\mathcal{E}(u_*^\shortparallel(t)) \in C^2(\bar \Om)$ with compact support
in $B_{R_1}$.

Decompose
\begin{equation}
\label{up-dec}
u = v+E, \quad E = \nb H +E_1, \quad p = \pi - \pd_t H.
\end{equation}

Then we have $E |_{\pd \Om} = u_*$, and $v, \pi$ satisfy
\begin{equation}
\label{S2:v-eq} v_t - \De v + \nb \pi = g(v),
\quad \div v = 0,
\quad v |_{\pd \Om} = 0,
\end{equation}
where
\begin{equation}
g(v) = \nb \cdot [F - (E+v)\ot ( E +v)] - \pd_tE_1 + \De E_1
\end{equation}
and $E$ satisfies $\|E\|_{L^{3/2,\I}} \le
\|u_*\|_{C^2(\pd \Omega)}$.
Moreover, since $\pd_t E_1$ is compactly supported in $\Omega$,
we can apply Lemma \ref{th:f.dec} so that
$\pd_t E_1$ can be written as
$$
\pd _t E_1=\nb F_1
\qquad {\rm with}
\qquad
\|F_1\|_{L^{q, \I}} \lec
\|\langle \cdot \rangle^a \pd_t E_1\|_{L^\I}
\lec \|\pd_t u_*\|_{C^2(\pd \Omega)}
$$
for $q \ge 3/2$ and $a>3$.
Therefore we can rewrite
$$
g(v) = \nb G(v) \qquad {\rm with} \qquad
\|G(v)\|_{\mathcal{X}^{3/2}} \lec \|F\|_{\mathcal{X}^{3/2}}
 +(\|u_*\|+ \|v\|_{\mathcal{X}^3})^2
+\|u_*\|.
$$
Here, $\|u_*\| = \|u_*\|_{W^{1,\I}(\R,C^2(\pd \Omega))}$,
and we do not distinguish $\nb$ and $\nb \cdot$ , since
the difference does not play any roles in our argument.

Following \cite{Yamazaki}, we consider the fixed point problem 
\begin{equation}\label{eq2.72}
v = \Phi v, \quad (\Phi v) (t) := \int_{-\I}^t e^{-(t-s)A} P \nb
G(v)(s)\,ds,
\end{equation}
where the integral is defined weakly
in the sense of Lemma \ref{th:2-1a}.

Applying Lemma \ref{th:2-1a}, we have
\EQ{
\|\Phi v  \|_{\mathcal{X}^3}
 \lec
\| G(v) \|_{\mathcal{X}^{3/2}}
\lec
\|F\|_{\mathcal{X}^{3/2}}
+(\|u_*\|+ \|v\|_{\mathcal{X}^3})^2 +\|u_*\|.
}

Similarly,
\begin{equation}
\begin{split}
\norm{\Phi v - \Phi \td v}_{ \mathcal{X}^3} &\lec (\norm{u_*} + \norm{ v}_{
\mathcal{X}^3}+\norm{\td v}_{ \mathcal{X}^3}) \norm{ v - \td v}_{ \mathcal{X}^3}.
\end{split}
\end{equation}
Thus, if
\begin{equation}
\begin{split}
\e= \norm{F}_{\mathcal{X}^{3/2} }  + \norm{u_*}
\end{split}
\end{equation}
is sufficiently small, there is a unique fixed point of $\Phi$ in
the class
\begin{equation}
\label{v-est1} \norm{v}_{\mathcal{X}^3} \lec
\e,\end{equation}
and $v \in BC_w(\R, L^{3,\I}_\si)$ by Lemma \ref{th:2-1a}.

The case $I = (0,\I)$ is proved similarly: We define $g_k(t)$ and
$E_1(t,x)$ as above for $t \ge 0$, and decompose $u=v+E$ as in
\eqref{up-dec}. The vector field $v(t,x)$ satisfies the initial
condition
\begin{equation}\label{v0.def}
v(0,x) = v_0(x) := u_0(x) - E(0,x).
\end{equation}
It is a fixed point of $\Phi$ where

\begin{equation}
\Phi v= e^{-tA} v_0 + \int_{0}^t e^{-(t-s)A} P \nb
G(v)(s)\,ds.
\end{equation}

Denote $\mathcal{X}_+^q = L^\I(\R_+,L^{q,\I})$. By similar estimates, there
is a unique fixed point of $\Phi$ in the class
\begin{equation} 
\norm{v}_{\mathcal{X}_+^3} \le C \e, \quad \e=\norm{v_0}_{L^{3,\I}} +
\norm{F}_{\mathcal{X}_+^{3/2} } + \norm{u_*},
\end{equation}
if $\e$ is sufficiently small, and $v \in
BC_w([0,\I), L^{3,\I}_\si)$ by Lemma \ref{th:2-1a}.

\subsection{Equivalence to a very weak solution}
\label{sec2.7}

We now verify that our solution $u=v+E$ is the unique very weak
solution in the class $u \in BC_w(\bar I,L^{3,\I}(\Om))$ of \eqref{NS1}
with the given data $u_0$, $u_*$, and $f$, with $I=\R$ or
$I=(0,\I)$. 

\medskip

{\bf A  mild solution is a very weak solution}. 
We first show that our solution is a very weak solution. Clearly $u$
satisfies \eqref{vws.def2} and $\div u=0$. It suffices to show
\eqref{vws.def}.

By divergence theorem, for $w \in \cD_I$, defined in \eqref{cDJ.def},
and fixed $t$,
\[
\ip{u_*}{N\cdot \nb
w}_{\pd \Om} = \ip{E}{N\cdot \nb
w}_{\pd \Om} =(E,\De w) + (\nb E, \nb w) = (E,\De w) - (\De E, w).
\]
Thus, upon writing $u=E+v$, \eqref{vws.def} is equivalent to
\begin{equation}
\int_I - \ip{v}{w_t+\De w} dt = \ip{v_0}{w(0)}- \int_I \ip{F_2}{\nb
  w} dt,\quad \forall w\in \cD_I
\end{equation}
for $F_2 =F_2(v)= F - (E+v)\ot (E+v) + \nb E_1 - F_1$. Recall $\nb F_1
= \pd_t E_1$. Above $v_0=0$ if $I=\R$, and $v_0$ is given by
\eqref{v0.def} if $I=(0,\I)$.

Let 
\begin{equation}
\cD :=\{ \psi \in C^2_c( \bar
\Om;\R^3), \div \psi=0, \psi|_{\pd \Om}=0\}.
\end{equation}
Note  $\nb\psi|_{\pd \Om} $ may not be zero for $\psi \in \cD$, and $\cD$ is a dense subset of $D(A)$.
Choosing $w(t,x) = \th(t)\psi(x)$ with $\th(t) \in C^1_c(\bar I)$ and
$\psi \in \cD$,
\eqref{vws.def} implies
\begin{equation}\label{eq2.82}
\int_I - \ip{v}{\psi} \th' dt = \ip{v_0}{\psi }\th(0)+\int_I \th
\bkt{\bka{v,\De \psi}-\ip{F_2}{\nb \psi}} dt.
\end{equation}
In turn this also implies \eqref{vws.def} since linear combinations of
such $w(t,x)= \th(t)\psi(x)$ with $\th \in C^1_c(\bar I)$ and $\psi \in \cD$ is
dense in $\cD_I$ in the norm $\sum_{a\le 1, b \le 2}\norm{\pd_t^a
  \nb_x^b w}_{C^0}$.%

Denote $t_0 = \inf I$.
Being a mild solution $v = \Phi v$  of \eqref{eq2.72} means $v(t)$ satisfies
\begin{equation}
(v(t),\psi) = (e^{-tA}v_0,\psi) +\int_{t_0}^t (-F_2(s),\nb e^{-(t-s)A} \psi)ds,\quad \forall
\psi \in L^{3/2,1}_\si(\Om), \forall t\in I.
\end{equation}
Plug this in the left side of
\eqref{eq2.82}. For $\psi \in \cD$ and $\th(t) \in C^2_c((t_0,\I))$, since $
      [\th'(t) - \e^{-1}(\th(t)- \th(t-\e))] \to 0$ as $\e \to 0_+$
      uniformly in $t$,
\begin{equation}\label{eq2.85}
\int_I - \ip{v}{\psi} \th' dt =  \lim_{\e \to 0+}
\e^{-1}\int_I  \ip{v(t)}{\psi} [-\th(t)+ \th(t-\e)] dt =
\end{equation}
\[
=
 \lim_{\e \to 0+}
\e^{-1}\int_I  \ip{v(t+\e) - v(t)}{\psi} \th(t) dt
= I_\e + \textit{II}_\e
\]
where
\begin{equation}\begin{split}
I_\e= &\int \th(t) (e^{-tA} v_0, \e^{-1}(e^{-\e A}-1)\psi)dt\\
&+\int \th(t) \int_{t_0}^t (-F_2(s),\nb e^{-(t-s)A} \e^{-1}(e^{-\e A}-1)\psi)ds\,dt,
\end{split}\end{equation}
\[
\textit{II}_\e = \e^{-1} \int \th(t)\int_t^{t+\e} (-F_2(s),\nb e^{-(t+\e-s)A}\psi)ds\,dt.
\]

For $\psi \in \cD \subset D(A)$, $- A\psi = P\De \psi \in L^{3/2,1}_\si$ and
\[
\norm{ \e^{-1}(e^{-\e A}-1)\psi + A\psi}_{L^{3/2,1}_\si} \to 0\quad \text{as}\quad
\e\to 0_+.
\]
Thus,  by Lemma \ref{th:2-1} and $v \in BC_w(\bar I; L^{3,\I}_\si )$,
\begin{equation}\begin{split}
I_\e \to & \int \th(t) (e^{-tA}v_0, P\De \psi)\,dt + \int \th(t) \int_{t_0}^t (-F_2(s),\nb e^{-(t-s)A} P\De \psi)ds\,dt 
\\
& = \int \th(t) (v(t), P\De \psi) dt = \int \th(t) (v(t), \De \psi) dt
,
\end{split}\end{equation}
as $\e\to 0+$.  Moreover,
with $\tau=(t+\e-s)$ and by Fubini theorem,
\[
\textit{II}_\e = \e^{-1} \int \th(t) \int_0^{\e} (-F_2(t+\e-\tau),\nb e^{-\tau
  A}\psi)d\tau\,dt
\]
\[
 =- \e^{-1}  \int_0^{\e} (\int \th(t) F_2(t+\e-\tau)dt,\, \nb e^{-\tau
  A}\psi) \, d\tau.
\]
Since $ \int \th(t) F_2(t+\e-\tau)dt = \int \th(t-\e+\tau) F_2(t)dt$
and $\th(t-\e+\tau) -\th(t) \to 0$ as $\e \to 0$ uniformly in $t$ and
$\tau\in [0, \e]$, we get
\[
\textit{II}_\e + \e^{-1}  \int_0^{\e} (\int \th(t) F_2(t)dt,\, \nb e^{-\tau
  A}\psi) \, d\tau \to 0
\]
as $\e \to 0_+$.  Since $e^{-\tau A}\psi$ is
continuous in $L^{3/2,1}_\si$, we get
\begin{equation}\label{eq2.94}
\lim_{\e \to 0_+} \textit{II}_\e = - (\int \th(t) F_2(t)dt,\, \nb\psi).
\end{equation}
The above shows \eqref{eq2.82} for $\psi \in \cD$ and $\th(t)\in
C^2_c((t_0,\I))$.
By approximation, \eqref{eq2.82} is also valid for $\psi \in
\cD$ and $\th(t)\in C^1_c([t_0,\I))$.

\medskip
 
{\bf A very weak solution is a mild solution}.  We next show that if
$u$ is a small very weak solution in the class $BC_w(\bar I,
L^{3,\I})$, then $v=u-E$ is a mild solution. This implies their
equivalence, and also the uniqueness of small very weak solutions in
the above class.

Let $u$ be small very weak solution in the class $BC_w(\bar I,
L^{3,\I})$, then $v=u-E$ satisfies \eqref{eq2.82}. Taking $\th(t) \in
C^1_c(I)$, \eqref{eq2.82} is the weak form of
\[
\frac{d}{dt}(v(t),\psi) = (v,\De \psi) - (F_2 ,\nb \psi).
\]
If we now take $\psi=\psi(t) = e^{-(t_1-t)A}\eta$ with $t \le t_1$ and
$\eta \in \cD$, we get $\frac {d}{dt}\psi(t) = -\De \psi(t)$ and
$\frac{d}{dt}(v(t),e^{-(t_1-t)A}\eta) = - (F_2 ,\nb
e^{-(t_1-t)A}\eta)$ weakly, that is 
\[
-\int (v(t),e^{-(t_1-t)A}\eta) \th'(t)dt = - \int  (F_2(t) ,\nb
e^{-(t_1-t)A}\eta)\th(t)\, dt
\]
for any $\th(t) \in C^1_c((\inf I, t_1))$. Take $t_0 \in [\inf I ,
t_1)$ and $\th(t) = \phi(\frac{t-t_1}{\e} +1) -
\phi(\frac{t-t_0}{\e})$ where $0<\e\ll 1$, $\phi(t) \in C^1(\R)$,
$\phi(t)=1$ for $t<0$ and $\phi(t)=0$ for $t>1$. Send $\e \to 0_+$. We
have $\th(t) \to 1_{t_0<t<t_1}$ and, by continuity of
$(v(t),e^{-(t_1-t)A}\eta)$,
\[
 (v(t_1),\eta)- (v(t_0),e^{-(t_1-t_0)A}\eta) = - \int_{t_0}^{t_1}  (F_2 (t),\nb
e^{-(t_1-t)A}\eta)\, dt.
\]

  If $I=\R$, we send $t_0 \to -\I$. Since $v \in
L^\I(\R;L^{3,\I})$ and $e^{-(t_1-t_0)A}\eta \to 0$ in $L^{3/2,1}_\si$
as $t_0 \to -\I$, we get
\[
 (v(t_1),\eta) = - \int_{-\I}^{t_1}  (F_2 (t),\nb
e^{-(t_1-t)A}\eta)\, dt, \quad \forall t_1\in I, \ \forall \eta \in \cD.
\]
If $I =(0,\I)$, we take $t_0\to 0_+$ and get
\[
 (v(t_1),\eta)- (v_0,e^{-t_1A}\eta) = - \int_{0}^{t_1}  (F_2 (t),\nb
e^{-(t_1-t)A}\eta)\, dt, \quad \forall t_1\in I, \ \forall \eta \in \cD.
\]
In either case $v(t)$ is a mild solution and is the unique one we
constructed in the previous subsection.

\subsection{Periodicity and spatial decay}
\label{sec2.8}

{\it Proof of part (ii)}.\quad Since $u(t+T,x)$ is another
solution of \eqref{NS1}--\eqref{NS2} with the same data and
estimates, we have $u(t+T,x)=u(t,x)$ by the uniqueness of part
(i).

\medskip

{\it Proof of part (iii)}.\quad  We now assume $I = \R$  and the
stronger assumption of (iii).

We first prove some a priori bounds. By Lemma \ref{th:2-1a}
again for a fixed $q \in (3/2,3)$ and $q^*$,
 we have
\begin{equation}
\begin{split}
\norm{v}_{\mathcal{X}^{q^*}}
&\lec  \norm{G(v)}_{  \mathcal{X}^{q}} \\
&\lec \norm{F}_{\mathcal{X}^{q}} + (\norm{E}_{ \mathcal{X}^{3}} +\norm{v}_{ \mathcal{X}^{3}})
(\norm{E}_{ \mathcal{X}^{q^*}} +\norm{v}_{  \mathcal{X}^{q^*}}) +\norm{u_*}.
\end{split}
\end{equation}
Hence, we obtain,
\begin{equation}
\label{v-est2} \norm{v}_{ \mathcal{X}^{q^*}} \lec \norm{F}_{ \mathcal{X}^{q}} +
\norm{u_*}
\end{equation}
when $\e$ is sufficiently small.

We now estimate $w= v_t$. It satisfies
\begin{equation}
w_t - \De w + \nb \pi_t =\nb \pd_t G,
\quad \div w = 0, \quad w |_{\pd \Om}
= 0,
\end{equation}
where
\begin{equation}
\pd_t G =
(F_t - (E_t +w)\ot (E+v)- (E+v) \ot (E_t +w)) -
\pd_t^2 F_1 + \pd_t \Delta E_{1}.
\end{equation}
For fixed $3/2< q<3$ we have
\begin{equation}
\norm{\pd_t E}_{ \mathcal{X}^{q^*}\cap \mathcal{X}^3} +
\norm{\pd_t^2 F_1}_{\mathcal{X}^{q} \cap \mathcal{X}^{3/2}}
+\norm{\pd_t \nb E_1}_{\mathcal{X}^q \cap \mathcal{X}^{3/2}}
\lec \norm{\pd_t u_*}_{W^{1,\I}(\R, C^2(\pd
\Om))}.
\end{equation}
Similar estimates show $v_t \in \mathcal{X}^3 \cap \mathcal{X}^{q^*}$ and
\begin{equation}
\norm{v_t}_{\mathcal{X}^3 \cap \mathcal{X}^{q^*}} \lec \e_1 :=\norm{F}_{W^{1,\I}(\R,
L^{3/2} \cap L^{q})} + \norm{u_*}_{W^{2,\I}(\R, C^2(\pd \Om))}.
\end{equation}

From now on we choose $q=2$ and thus $q^*=6$ for convenience.

Rewrite equation \eqref{S2:v-eq} for $v$ as a time-independent
Stokes system,
\begin{equation}
- \De v + \nb \pi =  g(v) - v_t, \quad \div v = 0.
\end{equation}
We have  $v \in \mathcal{X}^{3}\cap \mathcal{X}^{q^*}$, $F \in \mathcal{X}^{3/2} \cap \mathcal{X}^{q}$,
and $ - \pd_t E_1 + \De E_1- v_t \in \mathcal{X}^{3}\cap \mathcal{X}^{q^*}$.  
By bootstrapping as in \cite{KK06}, also see \cite{MT}, 
locally $E$ and $\pi-\phi(t)$ with a suitable $\phi(t)$ are bounded uniformly in
$t$. Let
\begin{equation}\label{Rcut.def}
\Rcut :=\{x\in \R^3:\quad R_1 < |x| < R_1+1\}.
\end{equation}
We now replace $\pi(t,x)$ by $\pi(t,x)-\phi(t)$ with a suitable
$\phi(t)$ so that
\begin{equation}
|v(t,x)|+ |\nb v(t,x)|+ |\pi(t,x)| \le C \e_1, \quad \forall t \in
\R, \quad \forall x \in \Rcut.
\end{equation}

Back to $u=E+v$ and $p = \pi-\pd_t H$, we have
\begin{equation}
\norm{u}_{\mathcal{X}^{3}\cap \mathcal{X}^{q^*}} \le C \td \e; \end{equation}
\begin{equation}  |u(t,x)|+ |\nb u(t,x)|+ |p(t,x)| \le C
\td \e, \quad \forall t \in \R, \quad \forall  x \in \Rcut.
\end{equation}
The pointwise decay estimate of (iii) follows from the following
lemma.\myendproof
\medskip

\begin{lemma}[Spatial decay of time-entire solutions]
\label{Spatial-decay} For any $R_1>0$, there are $\e_1>0$ and
$C>0$ such that the following holds. Suppose $u,p$ is a solution
of \eqref{NS1} 
with force
$f=f_0+\nb \cdot F$ for $R_1<|x|<\infty$ and $t \in \R$, and
satisfies
\begin{equation}
\begin{split} \e:= &\sup_{t\in \R} \bigg\{\sup _{|x|>R_1}
\bket{|x|^{2+\al} |f_0(t,x)| +|x|^{1+\al} |F(t,x)|} +
\norm{u(t,\cdot)}_{L^{3,\I}(|x|>R_1)} +
\\
&+  \sup _{R_1<|x|<R_1+1} \bket{|u(t,x)| +|\nb u(t,x)| +|p(t,x)| }
\bigg \}\le \e_1.\end{split}
\end{equation}
We do not assume any boundary condition at $|x|=R_1$. Then
\begin{equation}
|u(t,x)| \le C \e |x|^{-1}, \quad (|x| >R_1;\ t \in \R).
\end{equation}
\end{lemma}

{\it Proof}.\quad We perform a cut-off and the extend the solution
for $|x|>R_1$ to entire $\R^3$. Fix a smooth function $\zeta(x)$
which is 1 for $|x|>R_1+0.9$ and $0$ for $|x|<R_1+0.1$. Let
\begin{equation}\label{S2:td-u-eq1}
\td u = \zeta u + \nb \eta, \quad \div \td u =0,\quad \td p= \zeta
p-\pd_t \eta,
\end{equation}
where, for each $t$, $\eta$ is defined by the Newtonian potential so
that $\div \td u =0$, i.e.,
\begin{equation}\label{S2:eta.def}
\eta(t,x) = \int_{\Rcut} \frac {u(t,y) \cdot \nb
\zeta(y)}{4\pi|x-y|}dy, \quad -\De \eta =  u \cdot \nb \zeta.
\end{equation}
Note that $\tilde{u}, \td p$ are defined for $x\in \R^3$ and satisfy
the Stokes system for $(t,x) \in \R^{1+3}$
\begin{equation}\label{S2:td-u-eq2}
\pd_t \td u - \De \td u  + \nb \td p=\zeta f_0+\nb \cdot F_3(\td
u)+f_4, \quad \div \td u =0,
\end{equation}
where\footnote{Eq.~\eqref{S2:td-u-eq3} seems complicated because it
is used for the construction of $\td u$ without assuming decay
estimates of $u$, and hence cannot allow $u$ in $F_3$.}
\begin{equation}\label{S2:td-u-eq3}
\begin{split}
F_3(\td u) &=  \zeta  F - (\td u - \nb \eta)\ot (\td u - \nb
\eta),
\\
 f_4 &= -(\nb \zeta)\cdot F+\nb( u \cdot \nb \zeta)
  + \nb\cdot [(\zeta^2-\zeta)u \ot u)]
\\
&\qquad -2(\nb \zeta\cdot\nb) u -u\De \zeta + (u \cdot \nb \zeta)
u+p\nb \zeta . \end{split}\end{equation}
Note $F_3$ contains $\td u$ and global source terms, while
$f_4(t,x)$ contains only local source terms,
\begin{equation}\label{S2:td-u-eq4}
\supp f_4 \subset \R\times \Rcut,\quad \norm{f_4}_{L^\I_{t,x}} \le C
\e_1.
\end{equation}
Consider now the fixed point problem $w = \Phi w $ for the map
$\Phi$ from the class of vector fields defined on $\R \times \R^3$
into itself, defined by
\begin{equation}
(\Phi w) = \La(\zeta f_0+f_4) + \Th(F_3(w)).
\end{equation}
Here $\La$ and $\Th$ are defined by \eqref{Phi-def} and
\eqref{Psi-def}, respectively. We want to show it is a contraction
mapping in the class of small vector fields in $L^\I X_1$. By Lemma
\ref{th:Stokes},
\EQ{ \norm{\Phi w}_{L^\I X_1} &\le C\norm{\zeta f_0 + f_4}_{L^\I
X_{3+\de}} +C\norm{F_3}_{L^\I X_2}
\\
& \le C \e_1 + C (\e_1 +\norm{w}_{L^\I X_1})\norm{w}_{L^\I X_1} \\
\norm{\Phi w-\Phi \td w}_{L^\I X_1} & \le C (\e_1 + \norm{w}_{L^\I
X_1} +\norm{\td w}_{L^\I X_1} ) \norm{w-\td w}_{L^\I X_1}. }
Thus there is a constant $C_1$ so that $\Phi$ is a contraction
mapping in the class
\begin{equation}
\{ w(t,x):\R \times \R^3 \to \R^3, \quad \norm{w}_{L^\I X_1} \le
C_1 \e_1 \}.
\end{equation}
Thus there is a unique solution of $w=\Phi w$ in $L^\I X_1$.

By uniqueness of small solutions in $L^\I(\R, L^{3,\I}(\R^3))$, we
have $w= \td u = \zeta u + \nb \eta$. Thus
\begin{equation}
|u(t,x)| \le |w(t,x)| + |\nb \eta(t,x)| \le C_1\e \frac 1{|x|},
\quad (|x|>R_1+1).
\end{equation}

 \myendproof

\section{Spatial asymptotics of time periodic solutions}
\label{S3}

In this section we prove Theorem \ref{th2}. We start with a
$T$-periodic solution $(u,p)$ of \eqref{NS1}--\eqref{NS2} with force
$F$ for $|x|>R$, satisfying the estimates assumed in Theorem
\ref{th2}. Proceed as in the proof of Lemma \ref{Spatial-decay} and
let $R_1=R$, we fix a cut-off function $\zeta(x)$ with $\nb \zeta$
supported in $\Rcut$ and define
\begin{equation}
\td u= \zeta u + \nb \eta, \quad  \td p=p \zeta-\pd_t \eta, \quad
-\De \eta = \nb \zeta \cdot u,
\end{equation}
as in \eqref{S2:td-u-eq1} and \eqref{S2:eta.def}. Then $\td u$ and
$\td p$ satisfies \eqref{S2:td-u-eq2}--\eqref{S2:td-u-eq4}. By
Lemma \ref{Spatial-decay} we have $|u(t,x)|\le C \e |x|^{-1}$ for
$|x|>R$.

Let $b$ be the constant vector defined by \eqref{th2-eq2},
\begin{equation}
b_j = \lim_{\rho \to \I} \frac 1T\int_0^T \int _{|x|=\rho}
T_{ij}(u,p,F) n_i dS_x\,dt,\quad (n_i = \frac{x_i}{|x|},\ i=1,2,3),
\end{equation}
where $T_{ij}(u,p,F)=p\de_{ij} + u_i u_j - \pd_i u_j - \pd_j u_i -
F_{ij}$. By the fast spatial decay of $\nb \eta$, the time
periodicity, and divergence theorem,
\begin{equation}
b_j = \lim_{\rho \to \I} \frac 1T\int_0^T \int _{|x|=\rho}
T_{ij}(\td u,\td p,F) n_i dS_x\,dt = \frac 1T\int_0^T \int _{|x|\le
R_1+1} (\zeta f_0 + f_4)_j\, dx\,dt.
\end{equation}

Also let $U^b$ be the corresponding Landau solution, and let
\begin{equation}
\td U= \zeta U^b + \hat U, \quad \td P=\zeta P^b,
\end{equation}
where $\hat U$ has compact support in $\Rcut$, $\div \hat U = -\td U
\cdot \nb \zeta$, and
\begin{equation}
\|\hat U \|_{W^{3,6}}\le C\norm{\td U.\nb \zeta }_{W^{2,6}(\Rcut)}
\le C |b| \le C \e_1.
\end{equation}
Note that $\td U, \td P$ satisfies
\begin{equation}
 - \De \td U + \nb.(\td U \ot \td U) + \nb \td P=f_{\td U}, \quad
\div \td U =0,
\end{equation}
for $x \in \R^{3}$,  where \EQ{ f_{\td U} =& -\De \hat U -2\nb
U^b\cdot\nb \zeta -U^b\De \zeta +P\nb \zeta+ (\zeta^2-\zeta)\nb.(U^b
\ot U^b)
\\
&+\nb\cdot (\td U \ot \td U-\zeta U^b \ot \zeta U^b  ) +
(U^b\cdot\nb \zeta^2)U^b. }
\begin{equation}\label{fQ-est}
\norm{f_{\td U}}_{L^\I} \le C \e_1, \quad \supp f_{\td U} \subset
\Rcut.
\end{equation}
Moreover, \EQ{ \int_{\R^3} (f_{\td U})_j dx &= \int_{|x|\le R_1+2}
\pd_i T_{ij}(\td U,\td P) dx= \int_{|x|= R_1+2} T_{ij}(\td U,\td
P)\frac {x_i}{|x|} dS_x  \\ &= \int_{|x|= R_1+2} T_{ij}(U^b,
P^b)\frac {x_i}{|x|} dS_x = b_j. }

Let
\begin{equation}
v=\tilde{u}-\td U, \quad \pi= \tilde{p}-\tilde{P}.
\end{equation}
They satisfy, for $(t,x) \in \R^{1+3}$,
\begin{equation}\label{v-eq1}
\pd_t v - \De v  + \nb \pi= g+ \nb G(v), \quad \div v=0,\quad
|v(t,x)| \le C \e_2 \bka{x}^{-1}
\end{equation}
where (note $G(v)=F_3(\td u) + \td U\ot \td U $)
\begin{equation}\label{v-eq2}
\begin{split} g&=\zeta f_0 + f_4 - f_{\td U} , \\
G(v)&= \zeta F - (v - \nb \eta) \ot (v+\td U-\nb \eta) - \td U \ot
(v-\nb \eta) .\end{split}
\end{equation}

Because $ \int_{\R^3} (f_{\td U})_i dx= b_i$,
\begin{equation}\label{v-eq3}
\frac{1}{T}\int_0^T \int_{\R^3} g(t,x) \, dx\,dt =0.
\end{equation}

Consider now the fixed point problem $v = \Phi v $ for the map
$\Phi$ from the class of vector fields defined on $\R \times \R^3$
into itself, defined by
\begin{equation}
\Phi v = \La g+ \Th G(v).
\end{equation}
Here $\La$ and $\Th$ are defined by \eqref{Phi-def} and
\eqref{Psi-def}, respectively. We want to show it is a contraction
mapping in the class of small vector fields in $L^\I X_\al$,
$1<\al<2$. By Lemma \ref{th:Stokes},
\EQ{ \norm{\Phi v}_{L^\I X_\al} &\le C\norm{g}_{L^\I X_{3+\de}}
+C\norm{G(v)}_{L^\I X_{\al+1}}
\\
& \le C \e_1 + C \norm{\zeta F }_{X_\al}+ C\norm{v-\nb \eta
}_{X_\al} (\norm{v-\nb \eta }_{X_\al} + \norm{\td U }_{X_1})
\\
& \le C \e_1 + C (\e_1 +\norm{v}_{L^\I X_\al})\norm{v}_{L^\I X_\al}
\\
\norm{\Phi v-\Phi \td v}_{L^\I X_\al} & \le C (\e_1 + \norm{v}_{L^\I
X_\al} +\norm{\td v}_{L^\I X_\al} ) \norm{v-\td v}_{L^\I X_\al}. }
Thus there is a constant $C_2$ so that $\Phi$ is a contraction
mapping in the class
\begin{equation}
\{ v(t,x):\R \times \R^3 \to \R^3, \quad \norm{v}_{L^\I X_\al} \le
C_2 \e_1 \}.
\end{equation}
Thus there is a unique solution of $v=\Phi v$ in $L^\I X_\al$, which
agrees with $\td u - \td U$ by uniqueness of small solutions in
$L^\I(\R, L^{3,\I}(\R^3))$.

Since $v= \zeta u + \nb \eta - \zeta U^b - \hat U$, for $|x|>R_1+1$,
we have
\begin{equation}
|u(t,x)+\nb \eta(t,x) - U^b(x)| \le C_2 \e \bka{x}^{-\al}.
\end{equation}
Since $|\nb \eta(t,x)| \le C \bka{x}^{-2}$, we have proven Theorem
\ref{th2}.  \myendproof


\section{Perturbed Navier-Stokes flows}
\label{S4}

This section prepares a few lemmas for the proof of Theorem \ref{th4}
in \S\ref{S5}.  We first consider the solvability of the perturbed
Navier-Stokes system \eqref{w-eq0} in Proposition \ref{th3} and Lemma
\ref{th:int-est2}. We then prove a few estimates in Lemmas
\ref{th:4-2}--\ref{th:4-3}.

We now recall a few notions related to self-similar solutions.
The Navier-Stokes equations~\eqref{NS1} in $I \times \Omega  =
(0,\infty) \times \R^3$ with zero force enjoy the scaling property that if $u(t,x)$ is a solution,
then so are $u_\la(t,x):= \la u(\la^2 t,\la x)$ for any
$\la>0$. If $u_\la = u$ for all $\la>1$, then the solution is
called {\it
forward self-similar} (SS), and $u(t,x) = t^{-1/2} u(1,
t^{-1/2}x)$. If
 $u_\la = u$  for a particular $\la>1$, then the solution is called
{\it forward discretely self-similar} (DSS), and it is completely
decided by its values when $T\le t < T\la^2$ for any $T>0$. The
existence of both type of solutions for small initial data $u_0$
with $u_0(x)=\la u_0(\la x)$ for all $\la>1$ or a particular
$\la>1$, (also called SS or DSS), follows from Giga and Miyakawa
\cite{GM}, see also Cannone, Meyer and Planchon \cite{CMP}
and Cannone and Planchon \cite{CP}.

We now state our result on the solvability of
the perturbed Navier-Stokes system which is useful
to describe the time-asymptotics of solutions
for the initial-boundary value problem \eqref{NS1}--\eqref{NS3}
in the next section.
In particular, it implies the existence of the
self-similar solution in Theorem \ref{th4}.
\begin{proposition}[Perturbed system]\label{th3}
For any  $0<\eta<1$  there is $\e_0=\e_0(\eta)>0$ such that the
following holds. Let $U,\td U,w_0$ be vector fields in $\R^3$
satisfying $\sup_x |x|(|U(x)|+|\td U(x)|+ |w_0(x)|)\le \e \le
\e_0$ and $\div w_0=0$. Then there is a unique solution $w(t,x)$
of the perturbed Navier-Stoke system
\begin{equation}\label{w-eq}
\pd_t w -\De w + \nb (w \ot w + U \ot w + w \ot \td U) + \nb p =0,
\quad \div w = 0
\end{equation}
in $\R^3$ with initial data $w(0)=w_0$, satisfying
\begin{equation} \label{w-est}
|w(t,x)| \lec \e(|x|+\sqrt t)^{-1+\eta}|x|^{-\eta}.
\end{equation}
The pressure satisfies $\norm{p(t)}_{L^{s,\I}(\R^3)} \lec \e^2
t^{-1+\frac 3{2s}}$ for any $\frac 32 \le s \le \frac 3{1+\eta}$.

If furthermore $\sup_x |x|^2 (|\nb U(x)|+|\nb \td U(x)|+ |\nb
w_0(x)|)\le \e$, then
\begin{equation} \label{Dw-est}
|\nb w(t,x)|
\lec \e(|x|+\sqrt t)^{-1+\eta}|x|^{-1-\eta}.
\end{equation}
\end{proposition}

{\it Comments for Proposition \ref{th3}:}
\begin{enumerate}

\item $U$ and $\td U$ need not be divergence-free or self-similar.
If we assume in the following that both $U$ and $\td U$ are SS,
the solution set of \eqref{w-eq} has the same scaling symmetry
as the usual Navier-Stokes flows:
If $w(t,x)$ is a solution of \eqref{w-eq}, then so is
\begin{equation}\label{scaling-symmetry}
w_{\la}(t,x)=\la w(\la t, \ \la^2 x)
\end{equation}
for any $\la>0$. By this scaling symmetry and the uniqueness part
of Proposition \ref{th3}, $w$ is forward SS (or forward DSS) if
$w_0$ is SS (or DSS).

\item If $w$ is forward self-similar, then $w(t,x) = W(x/\sqrt t) /
\sqrt t$ with $W(x)=w(1,x)$. In the case $U=\td U=0$  and $w_0$ is
SS and small, we have $W\in L^{3,\I}$ and expect $|W(y)| \le \e
\bka{y}^{-1}$, i.e.,
\begin{equation} \label{S5eq:Q2}
|w(t,x)| \le \e (|x| + \sqrt t) ^{-1}.
\end{equation}

\item
For general initial data $|w_0(x)| < \e |x|^{-1}$ and $U=\td U=0$,
we can construct solution of \eqref{w-eq} satisfying the same
bound \eqref{S5eq:Q2}. However, this is impossible if $U$ or $\td
U$ is nonzero. See Remarks (i) and (ii) after the proof of
Proposition \ref{th3}.

\item In the case  $U$ and $\td U$ are nonzero, and all $U$, $\td U$ and $w_0$
are SS and small, $W(y)$ satisfies $\div W = 0$ and
\begin{equation}\label{WW-eq}
-\frac 12 W - \frac 12 x \cdot \nb W -\De W + \nb (W \ot W + U \ot
W + W \ot \td U) + \nb P =0.
\end{equation}
Since $W\in L^{3,\I}$, we expect $|W(y)|\lec \e|y|^{-1}$ for large
$y$. Due to the local singularity of $U$ and $\td U$, and the
local analysis of \cite{MT}, we can show $|W(y)|\lec \e
|y|^{-\eta}$ for small $y$ and some $\eta>0$, with a smaller $\e$
needed if we want a smaller $\eta$. Thus
\begin{equation}\label{WW-est}
 |W(y)| \le \e (|y|^\eta + |y|)^{-1},
\end{equation}
which corresponds to \eqref{w-est} for $w$.  For general data $w_0
\in X_1$, Proposition \ref{th3} asserts the unique existence of
solutions $w$ satisfying the bound \eqref{w-est}.

\end{enumerate}

In order to prove Proposition \ref{th3} we start with a lemma.

\begin{lemma} \label{th:int-est2}
Let $n \in \N$, $b \ge 0$, $c \ge 0$, $b+c<n$, $\mu>0$, $\la\ge
0$, and $t>0$. Then
\begin{equation}\label{th4.1-eq2}
\int_{\R^n} (|x-y|+\la)^{-b} |x-y|^{-c}(|y|+\sqrt t)^{-n-\mu}dy
\sim \sqrt t^{-\mu}(|x|+\la+\sqrt t)^{-b}(x+\sqrt t)^{-c} .
\end{equation}
\end{lemma}

\begin{proof} We may assume $t=1$ since the general case follows
from the change of variables $y=\sqrt t \hat y$, $x=\sqrt t \hat
x$ and $\la=\sqrt t \hat \la$. Denote the integral with $t=1$ as
$J(x)$ and $\rho=|x|+\la$. We want to show $J(x)\sim
\bka{\rho}^{-b} \bka{x}^{-c}$. Let $a=\mu+b+c$. If $\rho <10$, we
have
\begin{equation} J(x) \lec \int_{|y|<20} |x-y|^{-b-c}dy + \int_{|y|>20}
|y|^{-n-a} dy \lec 1.
\end{equation}
We also have $J(x) \gec  \int_{|y|>20} |y|^{-n-a} dy \gec 1$.
Assume now $\rho\ge  10$. We have
\begin{equation}
J(x)  = \int_{|y|<\rho/4} + \int_{|y|>\rho/4} =:J_1 + J_2.
\end{equation}

For $J_2$, by replacing the factor $(|y|+1)$ of the integrand by
$|y|$ and rescaling $y \to y/\rho$,
\begin{equation}
J_2 \le \rho^{-a}\int_{|y|\ge 1/4} (|\hat x-y|+\hat \la)^{-b}|\hat
x-y|^{-c}|y|^{-n-a+b+c} dy,
\end{equation}
where $\hat x = x/\rho$, $\hat \la = \la/\rho$ and $|\hat x| +\hat
\la =1$. The integral is of order 1 by the previous case
$\rho<10$. Thus $J_2 \lec \rho^{-a}$.

For $J_1$, because $| x-y|+ \la \ge \rho/4$ when $|y| \le \rho/4$,
\begin{equation}
J_1 \lec \rho^{-b}\int_{\R^n} |x-y|^{-c}\bka{y}^{-n-(a-b)+c} dy .
\end{equation}
Repeating the previous argument with $\la =0$ and $(b,a)$ replaced
by $(0,a-b)$, we get $J_1 \lec \rho^{-b} \bka{x}^{-c}$. We also
have $J_1 \gec \rho^{-b}\int_{|y|<5/2} |x-y|^{-c} dy \gec
\rho^{-b} \bka{x}^{-c}$.

Since $\rho^{-a}\lec \rho^{-b} \bka{x}^{-c} $, we get $J(x) \sim
\rho^{-b} \bka{x}^{-c}$.
\end{proof}

\bigskip

{\bf Proof of Proposition \ref{th3}:}

Equation \eqref{w-eq} can be written in the integral form using
the heat kernel $\Ga$ and the Stokes tensor \eqref{Stokes-tensor},
\begin{equation}
w=w_L+w_N(w),\quad
\end{equation}
where
\begin{equation}
w_L=e^{t\Delta}w_0, \quad w_L(t,x) =\int \Ga(t,y) w_{0,i}(x-y)dy,
\end{equation}
\begin{equation}
w_N(w)(t,x)=-\int_0^t\int_{\R^3} \pd_k S_{ij}(s,y)
F_{kj}(x-y,t-s)dyds,
\end{equation}
and $F=F(w)=w\otimes w+U \otimes w+w\otimes \td U$. Since
$e^{-p^2} \le C_k \bka{p}^{-k}$ for any $k>0$, we can take $k=4$
and have
\begin{equation}\Ga(t,y) \lec t^{-3/2} \bka{x/\sqrt t}^{-4} = \sqrt{t}(|x|+\sqrt
t)^{-4}.
\end{equation}
Suppose $|w_0(x)|\le \e(|x|+\la)^{-1}$, $\la=0,1$. We have
\begin{equation}
|w_L(t,x)| \lec \int  \sqrt t(|y|+\sqrt t)^{-4}\e (|x-y|+\la)^{-1}
dy.
\end{equation}
By Lemma \ref{th:int-est2} with $(n,\mu,b,c)=(3,1,1,0)$,
\begin{equation}\label{wL-est}
|w_L(t,x)| \lec \e (|x|+\la + \sqrt t)^{-1}.
\end{equation}
Suppose further $|\nb w_0(x)|\le \e(|x|+\la)^{-2}$.  By Lemma
\ref{th:int-est2} again with $(n,\mu,b,c)=(3,1,2,0)$,
\begin{equation}
|\nb w_L(t,x)| \lec \int  \sqrt t(|y|+\sqrt t)^{-4}\e
(|x-y|+\la)^{-2} dy \lec \e (|x|+\la + \sqrt t)^{-2}.
\label{DwL-est}
\end{equation}

For $0<\eta<1$, define two norms for functions on $\R_+\times
\R^3$:
\begin{equation}
\norm{f}_{\mathcal{Y}_1}=\sup_{(t,x)\in\R_+\times\R^3}\bkt{
(\abs{x}+\sqrt{t})^{1-\eta}\abs{x}^{\eta}\abs{f(t,x)}},
\end{equation}
\begin{equation}
\norm{f}_{\mathcal{Y}_2}=\norm{f}_{\mathcal{Y}_1}+\sup_{(t,x)\in\R_+\times\R^3}\bkt{
(\abs{x}+\sqrt{t})^{1-\eta}\abs{x}^{1+\eta}\abs{\nabla f(t,x)}}.
\end{equation}
Estimates \eqref{wL-est} and \eqref{DwL-est} show
$\norm{w_L}_{\mathcal{Y}_1} \le C_1\e$ (resp.~$\norm{w_L}_{\mathcal{Y}_2} \le C_1\e$)
if $\norm{w_0}_{X_1}\le \e$ (resp.~$\norm{w_0}_{X_1}+\norm{\nb
w_0}_{X_2}\le \e$) for some $C_1$.

We now estimate the nonlinear term $w_N(w)$. We will show
\begin{equation} \label{S4:wN-est1}
\norm{w_N(w)-w_N(\td w)}_{\mathcal{Y}_1} \lec  \e \norm{w-\td w}_{\mathcal{Y}_1}
\end{equation}
if $\norm{w}_{\mathcal{Y}_1}\le \e$ and $\norm{\td w}_{\mathcal{Y}_1}\le \e$, and
\begin{equation}\label{S4:wN-est2}
\norm{w_N(w)}_{\mathcal{Y}_2} \lec  \e \norm{w}_{\mathcal{Y}_2}
\end{equation}
if $\norm{w}_{\mathcal{Y}_2}\le \e$. Note \eqref{S4:wN-est1} implies
$\norm{w_N(w)}_{\mathcal{Y}_1} \lec  \e \norm{w}_{\mathcal{Y}_1}$ by taking $\td w=0$.
These two estimates imply that the map
\begin{equation}
w \to w_L + w_N(w)
\end{equation}
is a contraction mapping in the class of vector fields defined on
$\R_+ \times \R^3$ with $\norm{w}_{\mathcal{Y}_1} \le 2C_1 \e$
(resp.~$\norm{w}_{\mathcal{Y}_2} \le 2C_1 \e$) if $\norm{w_0}_{X_1}\le \e$
(resp.~$\norm{w_0}_{X_1}+\norm{\nb w_0}_{X_2}\le \e$) and $\e$ is
sufficiently small.

If $\norm{w}_{\mathcal{Y}_1}\le \e$, $\norm{\td w}_{\mathcal{Y}_1}\le  \e$, and
$\norm{U}_{X_1}+\norm{\td U}_{X_1}\le \e$, then
\begin{equation}\label{S4:F-est1}
|F(w)-F(\td w)|(t,x)\lec \e\norm{w-\td w}_{\mathcal{Y}_1}
(\abs{x}+\sqrt{t})^{-1+\eta}\abs{x}^{-1-\eta}.
\end{equation}
If $\norm{w}_{\mathcal{Y}_2}\le \e$ and $\norm{U}_{X_1}+\norm{\td
U}_{X_1}+\norm{\nb U}_{X_2}+\norm{\nb \td U}_{X_2}\le \e$, then
\begin{equation}
|\nb F(w)|(t,x)\lec \e^2
(\abs{x}+\sqrt{t})^{-1+\eta}\abs{x}^{-2-\eta}.
\end{equation}
Thus, using the definition of $w_N$ and $\nabla S$-estimate
\eqref{estimate-T}, to prove \eqref{S4:wN-est1} and
\eqref{S4:wN-est2} it suffices to show
\begin{equation}\label{S4:wN-est3}
\int_0^t \int_{\R^3} (|y|+\sqrt
s)^{-4}(\abs{x-y}+\sqrt{t-s})^{-1+\eta}\abs{x-y}^{-k-\eta}dy ds
\lec (\abs{x}+\sqrt{t})^{-1+\eta}\abs{x}^{1-k-\eta}
\end{equation}
for $k=1,2$. Decompose the integral as
\begin{equation}
I+II:=\int_0^{\frac{t}{2}}\int_{\R^3}\cdots dyds+
\int_{\frac{t}{2}}^t\int_{\R^3}\cdots dyds.
\end{equation}
We first estimate $I$. Since $(\abs{x-y}+\sqrt{t-s})\gec t$,
\begin{equation}
I \lec \sqrt{t}^{-1+\eta} \int_0^{\frac{t}{2}}\int_{\R^3}
(\abs{y}+\sqrt{s})^{-4}\abs{x-y}^{-k-\eta}dyds.
\end{equation}
By Lemma \ref{th:int-est2} with $(n,\mu,b,c)=(3,1,0,k+\eta)$ and
$t$ replaced by $s$,
\begin{equation}
I\lec \sqrt{t}^{-1+\eta} \int_0^{\frac{t}{2}}
\sqrt{s}^{-1}\bke{|x|+\sqrt s}^{-k-\eta}ds =
\frac{t^{\frac{-1+\eta}{2}}}{\abs{x}^{k+\eta-1}}\int_0^{\frac{t}{2x^2}}
\frac{d\tau}{\sqrt{\tau}(1+\sqrt{\tau})^{k+\eta}}.
\end{equation}
where we used the scaling $s=\abs{x}^2 \tau$. Using $\int_0^{T}
\frac{d\tau}{\sqrt{\tau}(1+\sqrt{\tau})^{k+\eta}}\lec \frac {
\sqrt T}{\sqrt T+1}$ for $k+\eta>1$, (thus we need $\eta>0$ for
$k=1$), we get
\begin{equation}
I\lec
\frac{t^{\frac{\eta}{2}}}{\abs{x}^{k+\eta-1}}\frac{1}{(\abs{x}+\sqrt{t})}
\end{equation}
which is bounded by
$(\abs{x}+\sqrt{t})^{-1+\eta}\abs{x}^{1-k-\eta}$.

Next we estimate $II$. Bounding the factor $(|y|+\sqrt s)^{-4}$ by
$(|y|+\sqrt t)^{-4}$ and then integrating in time,
\begin{equation}
II\lec  \int_{\R^3}(\abs{y}+\sqrt{t})^{-4} [
(\abs{x-y}+\sqrt{t/2})^{1+\eta} -\abs{x-y})^{1+\eta}]
\abs{x-y}^{-k-\eta}dy.
\end{equation}
Using $(A+B)^{1+\eta}-A^{1+\eta} \sim (A+B)^{\eta}B \sim
A^{\eta}B+B^{1+\eta}$ for $A,B,\eta\ge0$,
\begin{equation}
II \lec  \int_{\R^3}(\abs{y}+\sqrt{t})^{-4} [\abs{x-y})^{-k}\sqrt
t +\abs{x-y}^{-k-\eta}\sqrt{t}^{1+\eta}]dy.
\end{equation}
By Lemma \ref{th:int-est2} with $(n,a,b)=(3,1,0)$ and $c=k$ or
$c=k+\eta$,
\begin{equation}
II \lec (|x|+\sqrt t)^{-k},
\end{equation}
which is also bounded by
$(\abs{x}+\sqrt{t})^{-1+\eta}\abs{x}^{1-k-\eta}$. Summing up, we
have shown \eqref{S4:wN-est3} and thus the existence of $w$
satisfying \eqref{w-est} and \eqref{Dw-est}.

It remains to show the estimate of the pressure, which follows
from its equation
\begin{equation}
-\Delta
p=\partial_j\partial_{i}(w_iw_j+U_iw_j+w_iU_j)\qquad\mbox{in
}\,\,\R^3,
\end{equation}
the Calderon-Zygmund estimates, and $F(w)$-estimate in
\eqref{S4:F-est1} with $\td w=0$.  \myendproof

\bigskip

{\it Remarks.}\quad (i) If $U=\td U=0$ and $|w_0(x)|\le \e
(|x|+\la)^{-1}$ with $\la=0$ or $\la=1$, then one can construct
solutions in the class \begin{equation}|w(t,x)|\lec \e
(|x|+\la+\sqrt t)^{-1}.
\end{equation}
Indeed, in this case, we have $|F(t,x)|\lec \e^2 (|x|+\la+\sqrt
t)^{-2}$ and can bound $w_N(w)(t,x)$ by $C \e^2 (|x|+\la+\sqrt
t)^{-1}$ using
\begin{equation}
\int_0^t \int (|y|+\sqrt s)^{-4} (|x-y|+\la+\sqrt {t-s})^{-2} dy
ds \lec \sqrt t (|x|+\la+\sqrt t)^{-2}
.
\end{equation}
(ii) If $U$ or $\td U$ is nonzero, we need to take $\eta>0$ in
\eqref{w-est}. If $\eta=0$, we have $|F(t,x)|\lec \e^2 (|x|+\sqrt
t)^{-1}|x|^{-1}$ and need
\begin{equation}
I =\int_0^t \int_{\R^3} (|y|+\sqrt s)^{-4}
(|x-y|+\sqrt{t-s})^{-1}|x-y|^{-1}dyds \lec  (x+\sqrt t)^{-1}.
\end{equation}
However, when $|x| \ll \sqrt t$,
\begin{equation}
I \gec \int_0^{t/2}\int_{|y|<\sqrt t}(|y|+\sqrt s)^{-4}
(\sqrt{t})^{-1}|x-y|^{-1}dyds.
\end{equation}
Integrating in $ds$ first,
\begin{equation}
I \gec \int_{|y|<\sqrt t}|y|^{-2} (\sqrt{t})^{-1}|x-y|^{-1}dy .
\end{equation}
Restricted in the subregion $2|x|<|y|<\sqrt t$, 
\begin{equation} I
\gec \int_{2|x|<|y|<\sqrt t}|y|^{-3} (\sqrt{t})^{-1}dy\sim
(\sqrt{t})^{-1} \log \frac {\sqrt t}{2|x|}.
\end{equation}
It is larger than $(x+\sqrt t)^{-1}$ by a factor $\log \frac
{\sqrt t}{2|x|}$ when $|x| \ll \sqrt t$.

(iii) Since $\nb^2 F \not \in L^1_{x,loc}$ and $\nb^2 S \not \in
L^1_{x,t,loc}$, there is no suitable integral formula for $\nb^2
w$, and the above method does not allow us to estimate $\nb^2 w$
pointwise.

(iv) If all $U$, $\td U$ and $w_0$ are self-similar, then
$w(t,x)=W(x/\sqrt t)/\sqrt t$ with $W(y)$ satisfying the elliptic
equation \eqref{WW-eq}, and one can estimate higher derivatives of
$W$. It is not clear if $w_0$ is DSS.


\bigskip 

In the rest of this section, we give two lemmas to be used
in the next section.

\begin{lemma}\label{th:4-2}
Let $\Rcut$ and cut-off function $\zeta$ be as in section
\ref{S2b}. There is a linear map
\begin{equation}\label{La-domain-range}
\La: \bket{ w \in L^{1}_{loc}(\R^3;\R^3), \ \div w=0}  \to
\bket{\hat w \in L^1_{loc}(\R^3;\R^3), \ \supp \hat w\subset
\Rcut}
\end{equation}
such that, for any $1<q<\I$ and $q \le r\le \I$, $\hat w = \La w$
satisfies $ \div \hat w=-\nb \zeta \cdot w$ and
\begin{equation}
\norm{\hat w} _{W^{1,q,r}(\Rcut)} + \norm{\nb P\hat
w}_{L^{q,r}(\Om)}+ \norm{ P \hat w}_{L^{3/2,\I}(\Om) \cap
L^{q_b,\I}(\Om) } \le C_q \norm{w}_{L^{q,r}(\Rcut)}
\end{equation}
where $P$ is the Helmholtz projection on $L^{q,r}(\Om;\R^3)$, $q_b
= q^*$ if $q<3$, and $q_b=100$ if $q\ge 3$.
\end{lemma}

\begin{proof}
The usual construction of the solution of the problem $\div u = f$
(see \cite[Thm.~III.3.1]{Galdi}) gives such a linear map with
$\norm{\hat w}_{W^{1,q}(\Rcut)} \le C_q \norm{w}_{L^{q}(\Rcut)}$.
The bounds in Lorentz spaces follow from interpolation.
It remains to show $\norm{P\hat w}_{W^{1,q,r}(\Om)}\le C_q
\norm{w}_{L^{q,r}(\Rcut)}$. Decompose  $\hat{w}=P\hat w+\nabla\pi$
where the scalar function $\pi$ solves the following Neumann
boundary value elliptic problem:
\begin{equation}
\Delta \pi=-\nb \zeta \cdot w\quad \mbox{ in }\,\,\Omega,\qquad
\nabla\pi\cdot N=0 \quad \mbox{ on }\,\,\partial\Omega.
\end{equation}
Due to $L^p$ theory for elliptic equations with Neumann boundary
condition, (by partition of unity and boundary estimates in
\cite{ADN}, also see \cite{Smith}), we have
\begin{equation}\label{S4:pi-est}
\norm{\nabla^2\pi}_{L^{q,r}(\Omega)}\lesssim \norm{\nb \zeta \cdot
w}_{L^{q,r}(\Omega)}\lesssim \norm{w}_{L^{q,r}(\Rcut)}.
\end{equation}
Combining estimates,
\begin{equation}
\norm{\nabla P\hat w}_{L^{q,r}(\Om)}\lesssim \norm{\hat
w}_{W^{1,q,r}(\Omega)}
+\norm{\nabla^2\pi}_{L^{q,r}(\Omega)}\lesssim
\norm{w}_{L^{q,r}(\Rcut)}.
\end{equation}

To bound $\norm{ P \hat w}_{L^{3/2,\I}(\Om) \cap L^{q_b,\I}(\Om)
}$ amounts to bounding $\nb \pi$ is in the same space. This can be
shown by cut-off:  $u = \pi \zeta$ is in $W^{2,q,r}_{loc}$,
$C^2_{loc}$ outside of $\Rcut$, and solves $\Delta u = f$ for some
$f\in L^{q,r}$ with compact support. Thus $|\nb u(x)| \le C
|x|^{-2}$ for $|x|$ large using Newtonian potential. This
completes the proof.
\end{proof}

Note: Although $ f = \De u$ is a divergence, we do not know if
$\int f =0$ (which would imply extra decay for $u$) since $\De u$
is not integrable and we cannot use divergence theorem.

\begin{lemma}\label{th:4-3}

Let $w$ be the solution in Proposition \ref{th3} with $w_0$ and
$\nb w_0$ satisfying the stated estimates. Let $\hat w=\La w$ be
defined by Lemma \ref{th:4-2}. Then, for all $3 \le q \le q_1$,
\begin{equation}\label{hatw-est}
\norm{\hat w(t)}_{W^{1,q,\I}} + \norm{\int_0^t e^{-(t-s)A}P\pd_s
\hat w(s)ds }_{L^{q,\I}}\le C \e \bka{t}^{-\si}
\end{equation}
where $\si = \frac {1+\de}2-\frac 3{2q}>0$.
\end{lemma}

\begin{proof} Denote  $\mu=\frac{1-\eta}{2}$. Note $\si \le \mu$
if $q \le q_1$. By Lemma \ref{th:4-2} and Proposition \ref{th3},
\begin{equation} \norm{\hat
w(t)}_{W^{1,q,\I}} \lec \norm{ w(t)}_{L^{q,\I}(\Rcut)} \lec \norm{
w(t)}_{L^{\I}(\Rcut)} \lec \e \bka{t}^{-\mu}. \end{equation} For
the integral, using integration by parts, we have
\begin{equation}
\int_0^t e^{-(t-s)A}P\pd_s \hat w(s)ds =P\hat{w}(t)-e^{-tA}P \hat
w(0)-{\mathcal K}(t)
\end{equation}
where
\begin{equation}
{\mathcal K}(t)=\int_0^t
A^{\frac{1}{2}}e^{-(t-s)A}A^{\frac{1}{2}}P\hat w(s)ds.
\end{equation}
Choose $p$ so that $\mu=\frac 32(\frac 1{p^*}-\frac 1q)$, with
$\frac 1{p^*}=\frac 1{p}-\frac 13$. (We need $q>\frac 3{2(1-\mu)}
=\frac 3{1+\eta}$.) By Lemma \ref{th:4-2} and Proposition
\ref{th3},
\begin{equation}
\norm{P\hat{w}(t)}_{L^{q,\I}} \lec \norm{\nb
P\hat{w}(t)}_{L^{\frac {3q}{q+3},\I}} \le \e \bka{t}^{-\mu}.
\end{equation}
\begin{equation}
\norm{e^{-tA}P \hat w(0)}_{L^{q,\I}} \lec t^{-\si} \norm{P \hat
w(0)}_{L^{\qz ,\I}} \lec \e t^{-\si},
\end{equation}
while for $t<1$, $\norm{e^{-tA}P \hat w(0)}_{L^{q,\I}} \lec
\norm{P \hat w(0)}_{L^{q}} \lec \e$.

To estimate ${\mathcal K}(t)$, let $Z=\{ \varphi\in
L^{q',1}_{\sigma}: \norm{\varphi}_{L^{q',1}_{\sigma}}\leq 1\}$,
with $1/q+1/q'=1$ and $r=(q')^{*}$ i.e. $1/r=1/q'-1/3$. We have
\begin{equation}
t^{\mu}\norm{{\mathcal K(t)}}_{L^{q,\infty}}=\sup_{\varphi\in Z}\,
(t^{\mu}{\mathcal K}(t), \varphi)=\sup_{\varphi}\,
(t^{\mu}\int_0^t (s^{-\mu} A^{\frac{1}{2}} e^{-(t-s)A}\varphi,
s^{\mu} A^{\frac{1}{2}} P\hat{w}(s))ds
\end{equation}
\begin{equation}
\lesssim\sup_{\varphi}\, t^{\mu}\int_0^t s^{-\mu}
\norm{A^{\frac{1}{2}} e^{-(t-s)A} \varphi }_{L^{r,1}}
ds\,\sup_{s<t} s^{\mu}\norm{ A^{\frac{1}{2}}
P\hat{w}(s)}_{L^{r',\infty}}
\end{equation}
\begin{equation}
\lesssim \e\sup_{\varphi}\, t^{\mu}\int_0^t s^{-\mu}
\norm{A^{\frac{1}{2}} e^{-(t-s)A} \varphi }_{L^{r,1}} ds\lesssim
\e \norm{\varphi }_{L^{q',1}},
\end{equation}
where we used \eqref{th:4-1}.
This completes the proof.
\end{proof}


\section{Time asymptotics}
\label{S5}

In this section we consider large time asymptotics
of the solution close to the periodic solution.
We obtain Theorem \ref{th4}
as a consequence of a more general result
about the asymptotics:

\begin{theorem}[Time asymptotics]\label{th5}
For any $T>0$, $\de> 0$, $\eta>0$ and $3\le q_1 < \frac
3{\de+\eta}$, there is $\e_2>0$ such that
the following holds.
Let $u_*$ and $f$ be time-periodic data satisfying
\eqref{iii-est}.
Let $u_0 \in X_1$ be initial data satisfying
 $\|u_0\|_{L^{3,\I}} \le \ve_2$. Assume that
there exists a vector field $\td u_0 \in X_1$
such that
\begin{equation}
\label{u_0-est}
\|\tilde u_0 \|_{X_1} \le \ve_2, \quad
\nb \td u_0 \in X_2 \quad \textrm{and} \quad
\td \e :=\norm{u_0-\td u_0}_{L^{3,\I}\cap L^{\qz ,\I}}
 \le \e_2,
\end{equation}
then the solution $u$ in Theorem \ref{th1} (i)
can be decomposed as the follows:
$$
u=Q+w+r.
$$
Here $Q$ is the periodic solution for the data $u_*$ and $f$
in Theorem \ref{th1} (iii). The term
$w$ is the unique solution of the perturbed Navier-Stokes system
\eqref{w-eq} in $\R^3$ with $U=\td U = U^b$ and initial data $w_0=\td u_0-U^b$
in Theorem \ref{th3}, where
$U^b$ is the Landau solution corresponding to $Q$
given in Theorem \ref{th2} with $\al=1+\de$.
The term $r$ satisfies the following decay estimate:
\begin{equation*}
\norm{r(t)}_{L^q_w(\Om)}
\le C \td \ve t^{-\frac32(\frac13 - \frac 1q)-\frac{\de}{2}}
 \qquad  \forall t>0, \quad \forall q \in [3,q_1].
\end{equation*}

\end{theorem}

{\it Comments for Theorem \ref{th5}:}

\begin{enumerate}

\item
Theorem \ref{th4} is a special case of this theorem
when $\td u_0$ is self-similar (or (-1)-homogeneous).
We are also able to show that $w$ is DSS
when $\td u_0$ is DSS as we mentioned in the comments of
Proposition \ref{th3}.

\item
The restriction $q\ge 3$ is related to the fact on
the lack of the coercive estimate for the Stokes operator
in $L^{p,\I}(\Om)$ for large $p$ in the exterior domain.
See Lemma \ref{th:2-0A}. Note that it is not the case
when $\Om=\R^3$ or when $\Om$ is bounded.

\item By taking $\de+\eta$ small, we can choose $q_1<\I$ arbitrarily
large. But our method does not allow $q_1=\I$.

\item
Our method allows us to consider $u_* = \td u_* + \hat u_*$ where
$\td u_*$ is time periodic and $\hat u_*$ decays in time. We
assume $\hat u_*=0$ for simplicity.

\end{enumerate}

\noindent
\begin{proof}
\,
Suppose $\pd \Om$ belongs to the region $0<R_0 < |x|<R_1$ and let
$\Rcut$ and cut-off function $\zeta(x)$ be defined as in section
\ref{S2b}.
%
%
To prove Theorem \ref{th5}, we will study  $z=u-Q-\tilde{w}$
instead of $r=u-Q-w$ where $\td w$ is a cut-off of $w$ which
vanishes near $\pd \Om$:
\begin{equation}\label{tdw.def}
v= u-Q=\td w + z, \quad \td w =  \zeta w + \hat w,
\end{equation}
where $\hat w$ with compact support in $\Rcut$ and $\div \hat w =
- w\cdot \nb \zeta$ is given by Lemma \ref{th:4-2}. Since the
difference $w -\tilde{w}$ is localized near the boundary, it is
easy to show that it decays in time. Compared to
\eqref{S2:td-u-eq1}, this choice ensures $z|_{\pd \Om}=0$ but
introduces $\pd_t \hat w$ as a source term in $\pd_t z$ equation.

Specifically, $v$ and $w$ satisfy
\begin{equation}
\pd_t v - \De v + \nb p_1 = - \nb F_v, \quad \div v =0,
\end{equation}
\begin{equation}
\pd_t w - \De w + \nb p_0 = - \nb F_w, \quad \div w =0,
\end{equation} with
\begin{equation}
F_v = (v+Q)\ot v + v \ot Q, \quad F_w = (w+U^b)\ot w + w\ot U^b,
\end{equation}
and
\begin{equation}
v(0)=v_0:=u_0 - Q(0),\quad v|_{\pd \Om}=0,\quad w(0)=w_0:=\td u_0
- U^b.
\end{equation}
Thus, $\tilde{w}$ defined by \eqref{tdw.def} satisfies the Stokes
system for $(t,x) \in \R_+ \times \R^{3}$
\begin{equation}
\pd_t \wt w - \De \wt w + \nb (\zeta p_0)
=  -\nb (\zeta F_w+\nb
\hat w)+\tilde{f} + \pd_t \hat{w} , \quad \div \td w =0,
\end{equation}
with $\td w(0)=\td w_0:= \zeta w_0 + \hat w_0$, where
\begin{equation}\label{S4:td-u-eq3}
\begin{split}
\tilde{f} &= (\nb \zeta)\cdot F_w  -2(\nb \zeta\cdot\nb) w -(\De \zeta)w
+ p_0 \nb \zeta.
\end{split}\end{equation}
%

The vector field $z=v-\td w=u-Q-\td w$, defined by
\eqref{tdw.def},  satisfies, for $p_3=p_1 - \zeta p_0$,
\begin{equation}\label{5:z-eq1}
\pd_t z - \De z +  \nb p_3=\nb F_z - \tilde{f}  -\pd_t \hat{w}, \quad
\div z =0,\quad (x\in \Om)
\end{equation}
\begin{equation}\label{5:z-eq2}
z(0)=z_0 ,\quad z|_{\pd \Om}=0,
 \end{equation} where $F_z =\zeta F_w- F_v +\nb \hat w$ and
\begin{equation}
z_0 :=P(v_0-\td w_0) =P[(u_0 - \td u_0)+ (U^b - Q(0))+
((1-\zeta)w_0-\hat w_0)],
\end{equation}
and, using \eqref{th2-eq3} with $\al=1+\de$,
\begin{equation}
\norm{z_0}_{L^{3,\I}\cap L^{\qz ,\I}}\lec \e.
\end{equation}
Note we can add the Helmholtz projection $P$ in the definition of $z_0$ since $z(0)=z_0$ is understood in weak sense.
Since
\begin{equation}\begin{split}
F_v- F_w =&z\otimes z +(\td w +Q) \otimes z +z \otimes (\td w + Q)
+(Q-U^b) \otimes \td w + \td w \otimes (Q-U^b)
\\
&-(w-\td w) \otimes (w+U^b) -(\td w + U^b) \otimes (w- \td w),
\end{split}\end{equation}
we can decompose $F_z$ as
\begin{equation}\begin{split}
F_z=  z \otimes z +F_1(z) +F_2 +F_3 +F_4
\end{split}\end{equation}
where
\begin{equation}\begin{split}
&F_1(z) = z \otimes (\tilde w + Q) +(\tilde w +Q) \otimes z,
\\
&F_2 = \tilde w \otimes (Q-U^b) +(Q-U^b) \otimes \tilde w,
\\
&F_3 =(w-\td w)\otimes(w+U^b) -(\td w +U^b)\otimes (w - \td w),
\\
&F_4 =(\zeta -1) F_w +\nb \hat w.
\end{split}\end{equation}
Note that $F_3$ and $F_4$ have compact supports.

We will prove
\begin{equation} \label{5:z-est}
\sup_{t>0} t^{\sigma} \|z(t)\|_{L^{q, \infty}} \lec \ve \qquad
\mathrm{for} \quad q \in [3,q_1],
\end{equation}
with
\begin{equation}
\si = \si(q,\de)=\frac 32(\frac 13-\frac 1q)+\frac \de 2
%
>0.
\end{equation}
By unique existence of $u\in L^\I L^{3,\I}$ (Theorem \ref{th1})
and H\"older inequality, it suffices to prove the existence of
$z$ satisfying \eqref{5:z-eq1}, \eqref{5:z-eq2} and
\[
\norm{z}_{Z}:= \sup_{t>0} \{ t^{\de/2} \|z(t)\|_{L^{3, \infty}} +
t^{\si_1} \|z(t)\|_{L^{q_1, \infty}} \}\lec \e
\]
where $\si_1 = \si(q_1,\de)$.

Denote
\begin{equation}
\Psi (f)(t) =\int^t_0 e^{-(t-s)A} P  f(s)ds
\end{equation}
where the integral is in the weak sense,
\begin{equation}
\bke{\Psi (f)(t),\ph} = \int^t_0 (f(s),e^{-(t-s)A}\ph) ds,
\quad \forall \ph \in C^\I_{c,\si}(\Om).
\end{equation}

A solution $z(t)$ of \eqref{5:z-eq1} and \eqref{5:z-eq2} is a
fixed point of the nonlinear map $\Phi$ defined for $z\in Z$,
\[
\Phi z(t) =z^1(t)+ \Psi(\nb [z\ot z + F_1(z)])(t)
\]
where
\[
z^1(t)= e^{-tA}z_0 + \Psi(\nb [F_2+F_3+F_4] - \tilde{f} - \pd_t \hat
w)(t).
\]
We will show that for $\e$ sufficiently small,
\[ \label{fixedpoint-est}
\norm{\Phi z}_{Z} \le C \e, \quad \norm{\Phi z - \Phi \td z}_{Z}
\le C \e \norm{z- \td z}_{Z},
\]
which guarantees the unique existence of a fixed point $z$  of
$\Phi$ with $\norm{z}_{Z} \le C \e$.

By Lemma \ref{th:2-1} (i), 
we have
\begin{equation}
t^{\sigma} \|e^{-tA}z_0\|_{L^{q,\infty}} \lec
 \|z_0\|_{L^{\qz , \infty}}.
\end{equation}

Consider the terms of divergence form. By duality, we have
\begin{equation}\begin{split}
\|\Psi (\nb F)  \|_{L^{q,\I}} =\sup_{\ph \in L^{q', 1}_\sigma}
|\ip{\Psi(\nb F)}{\ph}|.
\end{split}\end{equation}
By \eqref{th:4-1} with $1/r=1/q'-1/3$ and $1/r'=1/q+1/3$, we
have
\begin{equation}
\begin{split}
|\ip{\Psi(\nb F)}{\ph}| &\le \int^t_0 \|F(s)\|_{L^{r',\infty}}
\|\nb e^{-(t-s)A}\ph \|_{L^{r,1}}ds
\\
&\lec  \bke{\sup_{t>0}t^{\sigma}\|F(t)\|_{L^{r',\I}}} \int^t_0
s^{-\sigma} \|\nb e^{-(t-s)A}\ph \|_{L^{r,1}}ds
\\
&\lec \bke{\sup_{t>0} t^{\sigma} \|F(t)\|_{L^{r',\I}}} t^{-\sigma}
\| \ph \|_{L^{q',1}}.
\end{split}\end{equation}
In the above we need $1<q' \le 3/2$, thus $3\le q <\I$.

We now estimate $F=F_z$. First,
\begin{equation}
t^{\sigma}\|z \otimes z +F_1(z) \|_{L^{r',\I}} \lec
t^{\sigma}\|z\|_{q,\I} (\|z\|_{3,\I}+\|\td
w\|_{3,\I}+\|Q\|_{3,\I})\lec \e^2.
\end{equation}
Next, since $w-\td w=(1-\zeta)w + \hat w$ is compactly supported,
it follows from Lemma \ref{th:4-2} and Proposition \ref{th3} that,
for $1<q<\I$,
\begin{equation}
\begin{split}
\|w-\td w\|_{L^{q,\I}(\Om)} &\lec \|w(1-\zeta)\|_{q,\I} + \|\hat w
\|_{q,\I}
\\
&\lec \|w(1-\zeta)\|_{q,\I} + \| w \|_{L^{q, \I}(\Rcut)}
\\
&\lec \|w \|_{L^\I(\Om)} \lec \e \langle t
\rangle^{-\frac{1-\eta}{2}}.
\end{split}\end{equation}
Also note
\begin{equation}
\norm{w(t)}_{L^{r,\I}(\R^3)} \lec \e t^ {-\frac 12 + \frac 3{2r}},
\quad (3 \le r \le 3/\eta).
\end{equation}
With $1/q_\flat=1/q-\delta/3$, (hence $\frac 12 - \frac
3{2q_\flat}=\si(q,\de)$), we have
\begin{equation}
\begin{split}\| F_2\|_{L^{r',\I}} &\lec \|Q-U^b\|_{\frac{3}{1+\delta},\I} \|\td
w\|_{q_\flat,\I}\lec \e(\|
w\|_{q_\flat,\I}+ \e \bka{t}^{-\frac{1-\eta}{2}}) \\
&\lec \e^2 (t^{-\si}+\bka{t}^{-\frac{1-\eta}{2}})\lec \e^2
t^{-\si} .
\end{split}\end{equation}
Recall  we take $\alpha =1+\delta$ in Theorem \ref{th2} so that
$Q-U^b \in L^{\frac{3}{1+\delta},\I}$. We have also used that $0<
\si \le \frac{1-\eta}{2}$ for $3 \le q \le q_1$, which also
implies
\begin{equation}
\|w-\td w\|_{L^{q,\I}(\Om)} \lec \ve t^{-\sigma}.
\end{equation}
Thus
\begin{equation}
t^{\sigma}\|F_3\|_{r',\I} \lec t^{\sigma} \|w-\td w\|_{q, \I}
(\|w\|_{3,\I}+\|\td w \|_{3,\I}+\|U^b\|_{3,\I}) \lec \ve ^2.
\end{equation}
Finally $F_4=(\zeta -1)F_w + \nb \hat w$ can be estimated as
follows:
\begin{equation}
\begin{split}
\|F_4\|_{r',\I} &\lec (\|w
\|_{3,\I}+\|U^b\|_{3,\I})\|(\zeta-1)w\|_{q, \I}
+\|w\|_{L^{r',\I}(\Rcut)}
\\
&\lec \|w \|_{L^\I(\Om)} \lec \e \langle t
\rangle^{-\frac{1-\eta}{2}} \lec \e t^{-\si} .
\end{split}\end{equation}

Next we consider $\Psi(\tilde{f})$. Denote  $r_1=\frac{3}{2+\delta
+\eta}$ and $r_2=\frac{3}{1+\eta}$. Note $1<r_1<r_2$. By  decay
estimates,
\begin{equation}
\|\Psi (\tilde{f})(t)\|_{q,\infty} \lec \int^t_0 (t-s)^{-\si_2}
\|\tilde{f}(s)\|_{r_1,\infty}ds, \quad \si_2 =\frac 32(\frac 1{r_1} -
\frac 1q).
\end{equation}
Note $0<\si_2<1$ since $q \le q_1< \frac 3{\de+\eta}$. By
Proposition \ref{th3}, on $\Rcut$ we have $|w|+|\nb w| \lec \e
\bka{t}^{-\frac 12 + \frac \eta 2}$, $|F_w|\lec |w|^2 + \e |w|
\lec \e ^2 \bka{t}^{-\frac 12 + \frac \eta 2}$ and $
 \|p_0(t)\|_{r_2,\infty}\le \e^2 \bka{t}^{-\frac 12 + \frac
\eta 2}$. Thus, using  $\supp \tilde{f}(t) \subset \Rcut$,
\begin{equation}
 \|\tilde{f}(s)\|_{r_1,\infty}\lec \|\tilde{f}(s)\|_{r_2,\infty}\lec \e \bka{s}^{-\frac 12 +
\frac \eta 2}.
\end{equation} We get
\begin{equation}
\|\Psi (\tilde{f})(t)\|_{q,\infty} \lec \int^t_0 (t-s)^{-\si_2} \e
\bka{s}^{-\frac 12 + \frac \eta 2} ds \lec \e t^{-\si_2+\frac 12 +
\frac \eta 2} =\e t^{-\si}.
\end{equation}

Finally the estimate $\Psi(\pd_t \hat w)$ is proved by Lemma
\ref{th:4-3}.

The above shows the first assertion $\norm{\Phi z}_{Z} \le C \e$
in \eqref{fixedpoint-est}. The second estimate in
\eqref{fixedpoint-est} on $\norm{\Phi z - \Phi \td z}_Z $ is
proved similarly. This finishes the proof of Theorem \ref{th5}.
\end{proof}

\section*{Acknowledgments}

We thank Professor C.-C. 
Chen for providing the references \cite{ADN,Smith} for
\eqref{S4:pi-est}, and Professor Y. Giga for discussing
Lemma~\ref{th:2-1} (ii).
Part of this work was done when all of us
visited the National Center for Theoretical Sciences (Taipei
Office) and the Department of Mathematics, National Taiwan
University, when both Kang and Miura visited the University of
British Columbia, and when Tsai visited Sungkyunkwan University
and Taida Institute for Mathematical Sciences. We would like to
thank the kind hospitality of these institutions and Professors
D.~Chae, C.-C. Chen, S.~Gustafson, J.~Lee, C.-S. Lin and C.-L.
Wang. The research of Kang is partly supported by
KRF-2008-331-C00024 and R01-2008-000-11008-0. The research of
Miura was partly supported by the JSPS grant no.~191437. The
research of Tsai is partly supported by Natural Sciences and
Engineering Research Council of Canada, grant no.~261356-08.


\begin{thebibliography}{10}

\bibitem{ADN}
Agmon, S.; Douglis, A.; Nirenberg, L.: Estimates near the boundary
for solutions of elliptic partial differential equations
satisfying general boundary conditions. Comm. Pure Appl. Math.
12 1959 623--727.

\bibitem{Amann} Amann, H.: 
Navier-Stokes Equations with
Nonhomogeneous Dirichlet Data, Journal of Nonlinear Mathematical
Physics Volume 10, Supplement 1 (2003), 1--11.


\bibitem{BM95} 
Borchers, W.; Miyakawa, T.:
On stability of
  exterior stationary Navier-Stokes flows.  Acta Math.  174 (1995),
  no. 2, 311--382.




\bibitem{CK}
Cannone, M.; Karch, G.:
   Smooth or singular solutions to the Navier-Stokes system?
       J. Differential Equations 197 (2004), no. 2, 247--274.

\bibitem{CMP}
Cannone, M.; Meyer, Y.; Planchon, F.: Solutions auto-simlaires des
equations de Navier-Stokes, In: Seminaire X-EDP, Centre de
Mathematiques, Ecole polytechnique, 1993-1994.

\bibitem{CP} Cannone, M.; Planchon, F.: Self-similar solutions for the
Navier-Stokes equations in $R^3$. Comm. Part. Diff. Equ. 21  (1996),  no. 1-2, 179--193.

\bibitem{CDW}
Cazenave, T., Dickstein, F.,Weissler, F.: Chaotic behavior of
solutions of the Navier-Stokes system in $R^N$ . Adv. Differ. Equ.
10 (4), 361-398, 2005.




\bibitem{FKS} 
Farwig, R.; Kozono, H.; Sohr, H.: Very
  weak solutions of the Navier-Stokes equations in exterior domains
  with nonhomogeneous data.  J. Math. Soc. Japan 59 (2007), no. 1,
  127--150.

\bibitem{Finn}
Finn, R.: On the exterior stationary problem for the
Navier-Stokes equations, and associated perturbation problems. Arch.
Rational Mech. Anal.  19  1965 363--406.

\bibitem{Galdi}   Galdi, G.~P.:
{An Introduction to the Mathematical Theory of the
Navier-Stokes Equations: Linearized steady problems}, Volume I,
Springer, 1994.

\bibitem{Galdi2}  Galdi, G.~P.:
{An Introduction to the Mathematical Theory of the
Navier-Stokes Equations: Nonlinear steady problems}, Volume II,
Springer, 1994.


\bibitem{Galdi-Sohr}  Galdi, G.; Sohr, H.: Existence and
uniqueness of time-periodic physically reasonable Navier-Stokes flow
past a body. Arch. Ration. Mech. Anal. 172 (2004), no. 3, 363--406.

\bibitem{GM}
Giga, Y.; Miyakawa, T.: Navier-Stokes Flow in $R^3$ with Measures
as Initial Vorticity and Morrey Spaces. Comm. Part. Differ. Equ.
14 (5), 577--618, 1989.

\bibitem{Heywood} Heywood, J.~G.: The exterior nonstationary problem
  for the Navier-Stokes equations. Acta Math. 129 (1972), no. 1-2,
  11--34.


\bibitem{Kap-Pil} Kapitanski\u{i}, L.~V.; Piletskas, K.~I.:
Some problems of vector analysis. (Russian) Boundary value
problems of mathematical physics and related problems in the
theory of functions, 16.  Zap. Nauchn. Sem. Leningrad. Otdel. Mat.
Inst. Steklov. (LOMI)   138 (1984), 65--85.



\bibitem{KK06} Kim, H.; Kozono, H.: A removable isolated
singularity
theorem for the stationary Navier-Stokes equations. Journal of
Differential Equations 220 (2006), 68--84.



\bibitem{Kor-Sve}
Korolev, A.;  Sverak,~V.: { On the large-distance asymptotics of
steady state solutions of the Navier-Stokes equations in 3D
exterior domains}, preprint: arXiv:0711.0560

\bibitem{KoNa}
Kozono, H.; Nakao, M.: Periodic solutions of the Navier-Stokes
equations in unbounded domains. Tohoku Math. J. (2) 48, 33--50
(1996).

\bibitem{KO}
Kozono, H.; Ogawa, T.: On stability of Navier-Stokes
flows in exterior domains. Arch. Rational Mech. Anal. 128 (1994),
no. 1, 1--31.


\bibitem{KY98MZ}
Kozono, H.; Yamazaki, M.: On a larger class of stable
solutions to the Navier-Stokes equations in exterior domains.  Math.
Z. 228 (1998), no. 4, 751--785.

\bibitem{KY09} Kozono, H.; Yanagisawa, T.: Nonhomogeneous boundary
  value problems for stationary Navier-Stokes equations in a multiply
  connected bounded domain.  Pacific J. Math. 243 (2009), no. 1,
  127-150.

\bibitem{Landau} Landau, L. D.:  A new exact solution of the
Navier-Stokes equations,  Dokl.\ Akad.\ Nauk SSSR,  43, 299, 1944.

\bibitem{LL} Landau, L. D.; Lifshitz, E. M.:  Fluid Mechanics,
second edition, Butterworth-Heinemann, 2000 paperback reprinting.

\bibitem{Mar-Pad} Maremonti, P.; Padula, M.: Existence, uniqueness
and attainability of periodic solutions of the Navier-Stokes
equations in exterior domains. (English, Russian summary) Zap.
Nauchn. Sem. S.-Peterburg. Otdel. Mat. Inst. Steklov. (POMI) 233
(1996), Kraev. Zadachi Mat. Fiz. i Smezh. Vopr. Teor. Funkts. 27,
142--182, 257; translation in J. Math. Sci. (New York) 93 (1999),
no. 5, 719--746.


\bibitem{MT} Miura, H.; Tsai, T.-P.:
Point singularities of 3D stationary Navier-Stokes flows, J. Math.
Fluid Mech.,
DOI: 10.1007/s00021-010-0046-6

\bibitem{Miyakawa} Miyakawa, T.: On nonstationary solutions of
the Navier-Stokes equations in an exterior domain.
Hiroshima Math. J.  12  (1982), no. 1, 115--140.

\bibitem{Naz-Pil}
Nazarov, S. A., Pileckas, K.: On steady Stokes and Navier-Stokes
problems with zero velocity at infinity in a three-dimensional
exterior domain, J. Math. Kyoto Univ. 40-3 (2000), 475--492.


\bibitem{Oseen} Oseen, C. W.: Hydrodynamik, Leipzig, 1927.

\bibitem{Pla}
Planchon, F.: Asymptotic behavior of global solutions to the
Navier-Stokes equations in $R\sp 3$.  Rev. Mat. Iberoamericana  14
(1998),  no. 1, 71--93.

\bibitem{Salvi} Salvi, R.: On the existence of periodic weak
solutions on the Navier-Stokes equations in exterior regions with
periodically moving boundaries. Navier-Stokes equations and
related nonlinear problems (Funchal, 1994), 63--73, Plenum, New
York, 1995.

\bibitem{Simader-Sohr} Simader,~C.~G.; Sohr,~H.: A new approach to the
  Helmholtz decomposition and the Neumann problem in Lq spaces for
  bounded and exterior domains. In: G.P. Galdi, Editor,
  Math. Probl. Relating to the Navier-Stokes Equations, World
  Scientific, Singapore (1992).

\bibitem{Smith}
Smith, H.~F.: The subelliptic oblique derivative problem. Comm.
Partial Differential Equations 15 (1990), no. 1, 97--137.

\bibitem{Solonnikov} Solonnikov, V. A.:
Estimates for solutions of a non-stationary linearized system of
Navier-Stokes equations. (Russian) Trudy Mat. Inst. Steklov. 70
(1964) 213--317.

\bibitem{Sverak} Sverak,~V.:  { On Landau's Solutions of
the Navier-Stokes Equations}, preprint: arXiv:math/0604550

\bibitem{Taniuchi}
Taniuchi, Y.: On the uniqueness of time-periodic solutions to
the Navier-Stokes equations in unbounded domains.  Math. Z.  261
(2009), no. 3, 597--615.

\bibitem{Tian-Xin}
Tian, G.; Xin, Z.: One-point singular solutions to the
Navier-Stokes equations. Topol. Methods Nonlinear Anal. 11 (1998),
no. 1, 135--145.

\bibitem{Yamazaki} Yamazaki, M.: The Navier-Stokes equations in the
  weak-$L\sp n$ space with time-dependent external force.  Math. Ann.
  317 (2000), no. 4, 635--675.

\end{thebibliography}
\end{document}